\documentclass[11pt,reqno]{amsart}
\usepackage{bbm}
\usepackage{graphics}
\usepackage{url}
\usepackage{hyperref}
\hypersetup{
    colorlinks=true,                          
    linkcolor=blue, % Couleur des liens internes
    citecolor=red, % Couleur des numéros de la biblio dans le corps
    urlcolor=blue  } % Couleur des url
\usepackage[utf8]{inputenc}
\usepackage{chngcntr}
\usepackage{lipsum}
\usepackage{amsthm}
\usepackage{amsfonts}
\usepackage{amsmath}
\usepackage{amssymb}
\usepackage{latexsym}
\usepackage{amssymb}
\usepackage{mathrsfs}
\usepackage{amsmath}
\usepackage{enumerate}
\usepackage{hyperref}
\usepackage{eurosym}
	\usepackage{cases}
\usepackage{color}
\newcommand{\kom}[1]{}
\renewcommand{\kom}[1]{{\bf [#1]}}

 \def\1{\raisebox{2pt}{\rm{$\chi$}}}
\def\a{{\bf a}}

\newtheorem{theorem}{Theorem}[section]

\newtheorem{lemma}[theorem]{Lemma}
\newtheorem{proposition}[theorem]{Proposition}
\newtheorem{definition}[theorem]{Definition}

\newtheorem{remark}[theorem]{Remark}

\numberwithin{equation}{section}

\newcommand{\R}{{\mathbb R}}

\newcommand{\N}{{\mathbb N}}

 \newcommand{\eps}{{\varepsilon}}
 \def\1{\raisebox{2pt}{\rm{$\chi$}}}
 
\newcommand{\abs}[1]{\left|#1\right|}
\newcommand{\norm}[1]{\left|\left|#1\right|\right|}
\newcommand{\Rn}{\mathbb{R}^n}
\newcommand{\osc}{\operatorname{osc}}

\def\vint_#1{\mathchoice%
          {\mathop{\kern 0.2em\vrule width 0.6em height 0.69678ex depth -0.58065ex
                  \kern -0.8em \intop}\nolimits_{\kern -0.4em#1}}%
          {\mathop{\kern 0.1em\vrule width 0.5em height 0.69678ex depth -0.60387ex
                  \kern -0.6em \intop}\nolimits_{#1}}%
          {\mathop{\kern 0.1em\vrule width 0.5em height 0.69678ex
              depth -0.60387ex
                  \kern -0.6em \intop}\nolimits_{#1}}%
          {\mathop{\kern 0.1em\vrule width 0.5em height 0.69678ex depth -0.60387ex
                  \kern -0.6em \intop}\nolimits_{#1}}}
\def\vintslides_#1{\mathchoice%
          {\mathop{\kern 0.1em\vrule width 0.5em height 0.697ex depth -0.581ex
                  \kern -0.6em \intop}\nolimits_{\kern -0.4em#1}}%
          {\mathop{\kern 0.1em\vrule width 0.3em height 0.697ex depth -0.604ex
                  \kern -0.4em \intop}\nolimits_{#1}}%
          {\mathop{\kern 0.1em\vrule width 0.3em height 0.697ex depth -0.604ex
                  \kern -0.4em \intop}\nolimits_{#1}}%
          {\mathop{\kern 0.1em\vrule width 0.3em height 0.697ex depth -0.604ex
                  \kern -0.4em \intop}\nolimits_{#1}}}

\newcommand{\aveint}[2]{\mathchoice%
          {\mathop{\kern 0.2em\vrule width 0.6em height 0.69678ex depth -0.58065ex
                  \kern -0.8em \intop}\nolimits_{\kern -0.45em#1}^{#2}}%
          {\mathop{\kern 0.1em\vrule width 0.5em height 0.69678ex depth -0.60387ex
                  \kern -0.6em \intop}\nolimits_{#1}^{#2}}%
          {\mathop{\kern 0.1em\vrule width 0.5em height 0.69678ex depth -0.60387ex
                  \kern -0.6em \intop}\nolimits_{#1}^{#2}}%
          {\mathop{\kern 0.1em\vrule width 0.5em height 0.69678ex depth -0.60387ex
                  \kern -0.6em \intop}\nolimits_{#1}^{#2}}}

\newcommand{\ol}{\overline}
\newcommand{\Om}{\Omega}

\newcommand{\dist}{\operatorname{dist}}

\newcommand{\diam}{\operatorname{diam}}

\newcommand{\tr}{\operatorname{tr}}
\renewcommand{\a}{\alpha}

\begin{document}

\title[$p$-Laplacian type problems in non-divergence form]{Remarks on regularity for $p$-Laplacian type equations in non-divergence form}
\author[Attouchi]{Amal Attouchi}
\address{Department of Mathematics and Statistics, University of
Jyv\"askyl\"a, PO~Box~35, FI-40014 Jyv\"askyl\"a, Finland}
\email{amal.a.attouchi@jyu.fi}

\author[Ruosteenoja]{Eero Ruosteenoja}
\address{Department of Mathematics and Statistics, University of
Jyv\"askyl\"a, PO~Box~35, FI-40014 Jyv\"askyl\"a, Finland}
\email{eero.ruosteenoja@jyu.fi}

\date{\today}

\begin{abstract}
We study a singular or degenerate equation in non-divergence form modeled on the $p$-Laplacian,
$$-|Du|^\gamma\left(\Delta u+(p-2)\Delta_\infty^N u\right)=f\ \ \ \ \text{in}\ \ \ \Omega.$$ We investigate local $C^{1,\alpha}$ regularity of viscosity solutions in the full range $\gamma>-1$ and $p>1$, and provide local $W^{2,2}$ estimates in the restricted cases where $p$ is close to 2 and $\gamma$ is close to 0.

\end{abstract}

\maketitle

\section{Introduction}
\label{sec:intro}
In this paper we study local regularity properties of solutions of the following nonlinear PDE,
\begin{equation}\label{genpl}
-|Du|^\gamma \Delta^N_p u:=-|Du|^\gamma\left(\Delta u+(p-2)\Delta_\infty^N u\right)=f\ \ \ \ \ \text{in}\ \ \Omega,
\end{equation}
where $\gamma\in (-1, \infty)$, $p\in (1, \infty)$,  $f\in C(\Omega)$, $\Omega\subset \R^n$ with $n\geq 2$, and $\Delta_\infty^N u:=\langle D^2u\frac{Du}{|Du|},\frac{Du}{|Du|}\rangle$ is the normalized infinity Laplacian. Equation \eqref{genpl}  is degenerate when $\gamma> 0$ and singular when $-1<\gamma\leq 0$, and it is appropriate to use the concept of viscosity solution. The operator $|Du|^\gamma \Delta^N_p u$ is a generalization of the standard $p$-Laplacian $\Delta_p u:=\text{div}\,(|D u|^{p-2}D u)$. As special cases we mention the variational $p$-Laplace equation $-\Delta_p u=f$ when $\gamma=p-2$, and the non-variational normalized $p$-Laplace equation $-\Delta^N_p u=f$ when $\gamma=0$.

Non-variational operators modeled by the $p$-Laplacian have gained increasing interest during the last 15 years. For $C^2$ domains with $W^{2, \infty}$ boundary data, Birindelli and Demengel \cite{birdem3,birdem2, birdemen1} showed global H\"older and local Lipschitz estimates for a class of fully nonlinear elliptic equations including equation \eqref{genpl} by using Ishii-Lions' method.  As a consequence of Harnack estimates, D\'avila, Felmer and Quaas \cite[Theorem 1.2]{dfqsing} proved H\"older estimates up to the boundary for the singular case $\gamma\leq 0$ and $p>1$ with H\"older continuous boundary data and domains satisfying a uniform exterior cone condition. 

Birindelli and Demengel \cite{birdemen2} and Imbert and Silvestre \cite{silimb} proved $C^{1, \alpha}$ regularity results for related equations. In \cite[Proposition 3.6]{birdemen2} the authors proved H\"older regularity of the gradient for solutions of \eqref{genpl} for $\gamma\leq 0$ and $p\geq 2$ by using approximations and a fixed point argument. In \cite{silimb} the authors used improvement of flatness to show local $C^{1, \alpha}$ regularity in the degenerate case $\gamma>0$ for viscosity solutions of the equation $|Du|^\gamma \left(F(x, D^2 u)\right)=f$, where $F$ is a uniformly elliptic operator, and the result was extended to the full range $\gamma>-1$ in \cite{birdem5}. Notice that this result does not cover equation \eqref{genpl} because of discontinuous gradient dependence. In \cite{atparuo} with Parviainen, we showed local $C^{1,\alpha}$ regularity of solutions of \eqref{genpl} when $\gamma=0$ and $p>1$. For the special case of radial solutions, Birindelli and Demengel \cite{ birdemen1} showed $C^{1,\alpha}$ regularity of solutions of \eqref{genpl} for $\gamma>-1$ and $p>1$.

In this paper we extend these regularity results to viscosity solutions of \eqref{genpl} when $\gamma>-1$ and $p>1$, and provide some results on the existence and the integrability of the second derivatives. Our first result concerns the interior  regularity of the gradient.

\begin{theorem}\label{th1}
Let $\gamma>-1$, $p>1$, and $f\in L^{\infty}(\Omega)\cap C(\Omega)$. Then there exists $\alpha=\alpha(p,n, \gamma)>0$ such that any viscosity solution $u$ of
  \eqref{genpl} is in $C^{1,\alpha}_{\text{\emph{loc}}}(\Omega)$, and for any $\Omega' \subset \subset \Omega$,
$$
[u]_{C^{1,\alpha}(\Omega')} \le C=C \left(p,n,d, \gamma, d',||u||_{L^\infty(\Omega)},\norm{f}_{L^\infty(\Omega)} \right),
$$
where $d=\diam(\Om)$ and $d'=\dist(\Om', \partial\Om)$.
\end{theorem}

An iterative argument leads us to analyze compactness and regularity of deviations of solutions $u$ from planes, $w(x)=u(x)-q\cdot x$ for some $q\in \R^n$, and the key estimate for these deviations is called improvement of flatness, Lemma \ref{flatle}. A key step in the proof is to obtain Arzel\`a-Ascoli type compactness for deviations $u(x)-q\cdot x$. The main difficulty is that for $\gamma\neq 0$, the ellipticity constants of the equation satisfied by $w$ depend on $|q|$. To overcome this problem, we use both Ishii-Lions' method and Alexandrov-Bakelman-Pucci (ABP for short) estimate in the proofs of Lemmas \ref{prop1}, \ref{lip1}, and \ref{lipe2}.

Our second result concerns integrability of the second derivatives when $\gamma$ is negative and the range of $p$ is restricted.

\begin{theorem}\label{th2}
Let $\gamma\in(-1,0]$, $p\in (1,3+\frac{2}{n-2})$, and $f\in L^{\infty}(\Omega)\cap C(\Omega)$. Then any viscosity solution $u$ of \eqref{genpl} belongs to $W^{2,2}_{\text{\emph{loc}}}(\Omega)$, and for any $\Omega''\subset \subset \Omega'\subset \subset \Omega$, we have
\begin{equation}
\norm{u}_{W^{2,2}(\Omega'')}\leq C=C(p,n,\gamma,d',d'',\norm{u}_{L^\infty(\Omega)},\norm{f}_{L^\infty(\Omega)}), 
\end{equation}
where $d'=\diam(\Om')$ and $d''=\dist(\Om'', \partial\Om')$.
\end{theorem}

The proof starts from the observation that when $\gamma\leq 0$, viscosity solutions of \eqref{genpl} are viscosity solutions of the  inhomogeneous normalized $p$-Laplace equation with a drift, $-\Delta^N_p u=f|D u|^{-\gamma}$, see Lemma \ref{reducsing}. This equation is singular but uniformly elliptic. The main idea is to regularize the equation (see equation \eqref{perti}) and use the \emph{Cordes condition}, see Theorem \ref{cordes}. This condition guarantees a uniform estimate for solutions of regularized problems, which together with equi-H\"older continuity and compactness argument gives a local $W^{2,2}$ estimate for solutions of equation \eqref{genpl}.

Obtaining integrability for second derivatives in the case $\gamma>0$ is harder because of the degeneracy of equation \eqref{genpl}. However, for a small $\gamma$ and $p$ close to two, we obtain local $W^{2,2}$ estimate by again using an approximation with uniformly elliptic regularized problems, and considering equation \eqref{genpl} as a perturbation of the $(\gamma+2)$-Laplacian.

\begin{theorem}\label{th3}
Assume that $f\in W^{1,1}(\Omega)\cap L^{\infty}(\Omega)\cap C(\Omega)$ and $\gamma\in (0,\beta]$, where $\beta\in (0,1)$ and
\begin{equation}
1-\beta-\sqrt n|p-2-\gamma|-\beta(p-2-\gamma)^+>0.
\end{equation}
Then any viscosity solution $u$ of \eqref{genpl} belongs to $W^{2,2}_{\text{\emph{loc}}}(\Omega)$, and  for any $\Omega''\subset \subset \Omega'\subset \subset \Omega$, we have
\begin{equation}
\norm{u}_{W^{2,2}(\Omega'')}\leq C(p,n,\gamma,\beta,d',d'',\norm{u}_{L^\infty(\Omega)},\norm{f}_{L^\infty(\Omega)},\norm{f}_{W^{1,1}(\Omega)}),
\end{equation}
where $d'=\diam(\Om')$ and $d''=\dist(\Om'', \partial\Om')$.
\end{theorem}

This paper is organized as follows. In Section \ref{prel} we define viscosity solutions for the problem \eqref{genpl}.
In Section \ref{chapter3} we prove Theorem \ref{th1}, and in Section \ref{chapter4} Theorems \ref{th2} and \ref{th3}. \\

\noindent \textbf{Acknowledgment.} A.A. is supported by the Academy of Finland, project number 307870.

\section{Preliminaries and definitions of solutions}\label{prel}

In this section we define viscosity solutions of equation \eqref{genpl} and fix the notation.
For $\gamma>0$, the operator $|Du|^\gamma\left(\Delta u+(p-2)\Delta_\infty^N u\right)$  is continuous, and we can use the standard definition of viscosity solutions, see \cite{ishiilions}.
For $\gamma\leq 0$, the  operator  \eqref{genpl} is undefined when $D u=0$,
where  it has  a bounded discontinuity when $\gamma=0$ and is very singular when $\gamma<0$. This can be remediated adapting
the notion of viscosity solution using the upper and
lower semicontinuous envelopes (relaxations) of the operator, see \cite{crandall1992user}.
 The definition for viscosity solutions for the normalized $p$-Laplacian ($\gamma= 0$) is the following.
 
\begin{definition}
Let $\Om$ be a bounded domain and $1<p<\infty$. An upper semicontinuous function $u$ is  a viscosity subsolution of the equation $-\Delta^N_p u=f$
if for all $x_0\in\Om$  and $\phi\in C^2(\Om)$ such that $u-\phi$ attains a local maximum at $x_0$, one has
\begin{equation*}
\begin{cases}
-\Delta_p^N \phi(x_0)\leq f(x_0), &\text{if}\quad D \phi(x_0)\neq 0,\\
-\Delta\phi(x_0)-(p-2)\lambda_{max} (D^2\phi(x_0))\leq f(x_0),\, &\text{if}\ D \phi(x_0)=0\,\,\text{and}\,\, p\geq 2,\\
-\Delta\phi(x_0)-(p-2)\lambda_{min} (D^2\phi(x_0))\leq f(x_0),\, &\text{if}\ D \phi(x_0)=0\,\,\text{and}\,\, 1<p<2.
\end{cases}
\end{equation*}
 A lower semicontinuous function $u$ is a viscosity supersolution of \eqref{genpl}
if for all $x_0\in\Om$  and $\phi\in C^2(\Om)$ such that $u-\phi$ attains a local minimum at $x_0$, one has
\begin{equation*}
\begin{cases}
-\Delta_p^N \phi(x_0)\geq f(x_0), &\text{if}\quad D \phi(x_0)\neq 0,\\
-\Delta\phi(x_0)-(p-2)\lambda_{min} (D^2\phi(x_0))\geq f(x_0),\, &\text{if}\ D \phi(x_0)=0\,\text{and}\,\, p\geq 2,\\
-\Delta\phi(x_0)-(p-2)\lambda_{max} (D^2\phi(x_0))\geq f(x_0),\, &\text{if}\ D \phi(x_0)=0\,\text{and}\, 1<p<2.
\end{cases}
\end{equation*}
We say that $u$ is a viscosity solution of \eqref{genpl} in $\Om$ if it is both a viscosity sub- and
supersolution.
\end{definition}

The following definition for the case $\gamma<0$ is adapted from the definition used by Julin and Juutinen in \cite{julin2012new} for the singular $p$-Laplacian.
\begin{definition}\label{def1}
Let $\Om$ be a bounded domain, $1<p<\infty$, and $-1<\gamma< 0$. An upper semicontinuous function $u$ is  a viscosity subsolution of \eqref{genpl} if $u\not\equiv \infty$ and 
if for all $\phi\in C^2(\Om)$ such that $u-\phi$ attains a local maximum at $x_0$  and $D\phi (x)\neq 0$ for $x\neq x_0$, one has
\begin{equation}
\underset{r\to 0}{\lim}\,\underset{\substack{x\in B_r(x_0),\\ x\neq x_0}}{\inf}\, (-|D\phi(x)|^\gamma\Delta_p^N\phi(x))\leq f(x_0).
\end{equation}
 A lower semicontinuous function $u$ is a viscosity supersolution of \eqref{genpl}
if $u\not\equiv \infty$ and  for all $\phi\in C^2(\Om)$ such that $u-\phi$ attains a local minimum at $x_0$ and $D\phi (x)\neq 0$ for $x\neq x_0$, one has
\begin{equation}
\underset{r\to 0}{\lim}\, \underset{\substack {x\in B_r(x_0),\\ x\neq x_0}}{\sup}\,\left( -|D\phi(x)|^\gamma\Delta_p^N\phi(x)\right)\geq f(x_0).
\end{equation}
We say that $u$ is a viscosity solution of \eqref{genpl} in $\Om$ if it is both a viscosity sub- and
supersolution.
\end{definition}

Another definition of solutions was proposed by Birindelli and Demengel in \cite{birdem1,birdem3, birdem2}. This definition  of solution is an adaptation to the elliptic case of the notion of solution that was introduced by Chen, Giga
and Goto \cite{chen} and Evans and Spruck \cite{evas} for singular problems. This is a 
variation of the usual notion of viscosity solution for \eqref{genpl} that avoids  testing with  functions having vanishing gradient at the
testing point. 
\begin{definition}\label{def2}
Let $-1<\gamma< 0$ and $p>1$. A lower semicontinuous function $u$ is  a viscosity super-solution of \eqref{genpl} in $\Omega$, if for every $x_0\in \Omega$ one of the following conditions hold.
\begin{enumerate}[i)]
\item
Either for all $\phi \in C^2(\Omega)$ such that $u-\phi$ has a local minimum at $x_0$ and $D\phi(x_0)\neq 0$, we have
$$ -|D\phi(x_0)|^\gamma\Delta_p^N\phi(x_0)\geq f(x_0).$$
\item Or there is an open ball $B(x_0, \delta)\subset\Omega$, $\delta>0$  such that $u$ is constant in $B(x_0, \delta)$ and
$$0\geq f(x)\quad \text{for all}\ \  x\in B(x_0, \delta).$$
\end{enumerate}

An upper  semicontinuous function $u$ is a viscosity sub-solution  of \eqref{genpl} in $\Omega$ if for
all $x_0\in \Omega$ one of the following conditions hold.
\begin{enumerate}[i)]
\item Either for all $\phi\in C^2(\Omega)$ such that $u-\phi$ reaches a local maximum   at $x_0$ and $D\phi(x_0)\neq 0$, we have
\begin{equation}
-|D\phi(x)|^\gamma\Delta_p^N\phi(x))\leq f(x_0).
\end{equation}

\item
Either there exists an open ball $B(x_0, \delta)\subset\Omega$, $\delta>0$ such that $u$ is constant in $B(x_0, \delta)$ and 
\begin{equation}
0\leq f(x)\quad \text{for all}\ \ x\in B(x_0, \delta).
\end{equation}
\end{enumerate}
\end{definition}

\begin{proposition}
Definitions \ref{def1} and \ref{def2} are equivalent.
\end{proposition}

\begin{proof}
We show first that if $u$ is a supersolution of \eqref{genpl} in the sense of Definition \ref{def1}, then it is a supersolution in the sense of  Definition \ref{def2}. The case $i)$ is immediate, so we only need to consider the case where $u$ is constant in a ball $B(x_0, \delta)$. We see that for all $y\in B(x_0, \delta)$, the function $\phi(x)=u(y) 
-|x-y|^q$ with $q>\dfrac{\gamma+2}{\gamma+1}>2$ is a smooth test function with an isolated critical point, and it holds 
$$\underset{r\to 0}{\lim}\underset{\substack{z\in B_r(y),\\ z\neq y}}{\sup}\, -|D\phi(z)|^\gamma\Delta_p^N\phi(z)=0,$$
so that $f\leq 0$ in $B(x_0, \delta)$, and the statement follows.

Now we assume that $u$ is a supersolution in the sense of Definition \ref{def2}. Let $x_0\in \Omega$ and $\phi$ an admissible test function such that $u-\phi$ has a local minimum at $x_0$.  We assume without a loss of generality, that $x_0$ is a strict local
minimum in $B(x_0, \eta)$. If $D\phi(x_0)\neq 0$, then the conclusion is immediate. If $D \phi(x_0)=0$ and $\phi $ has only one critical point in a neighborhood  $B(x_0, r_0)$ of $x_0$, we have to consider two cases: either $u$ is constant or not.  If $u$ is constant in a small ball around $x_0$, then $u$ is smooth and it follows that $D^2\phi (x_0)\leq 0$. Using that
$$-\Delta_p^N\phi(x)\geq -\Delta \phi(x)-(p-2)\lambda_{\max}(D^2\phi(x))\qquad\textrm{if}\ \  p\geq 2, $$
$$-\Delta_p^N\phi(x)\geq -\Delta \phi(x)-(p-2)\lambda_{\min}(D^2\phi(x))\qquad\textrm{if}\ \ p< 2, $$
we get that for $p\geq 2$
$$\underset{r\to 0}{\lim}\underset{\substack{x\in B_r(x_0),\\ x\neq x_0}}{\sup}\, -\Delta_p^N\phi(x)\geq 
\underset{r\to 0}{\lim}\underset{\substack{x\in B_r(x_0),\\ x\neq x_0}}{\sup}\, -\Delta\phi(x)-(p-2)\lambda_{\max}(D^2 \phi(x))\geq 0,$$
and for $1<p<2$
$$\underset{r\to 0}{\lim}\underset{\substack{x\in B_r(x_0),\\ x\neq x_0}}{\sup}\, -\Delta_p^N\phi(x)\geq 
\underset{r\to 0}{\lim}\underset{\substack{x\in B_r(x_0),\\ x\neq x_0}}{\sup}\, -\Delta\phi(x)-(p-2)\lambda_{\min}(D^2 \phi(x))\geq 0,$$
where we used that $$-\Delta \phi(x_0)-(p-2)\lambda_{\max}(D^2\phi(x_0))\geq 0$$  and 
$$-\Delta \phi(x_0)-(p-2)\lambda_{\min}(D^2\phi(x_0))\geq 0.$$
It follows that
$$\underset{r\to 0}{\lim}\underset{\substack{x\in B_r(x_0),\\ x\neq x_0}}{\sup}\, -|D\phi(x)|^\gamma\Delta_p^N\phi(x)\geq 0.$$
Since we have $f\leq 0$  at $x_0$, it follows that 
$$\underset{r\to 0}{\lim}\underset{\substack{x\in B_r(x_0),\\ x\neq x_0}}{\sup}\, -|D\phi(x)|^\gamma\Delta_p^N\phi(x)\geq f(x_0).$$
 
 Next we consider the case where $u$ is not constant in any ball around  $x_0$. We use the argument  of \cite[Lemma 2.1]{dfqsing}. For $y\in B(0, r)$ where $r>0$ is small enough, we consider the function $\phi_y(x)= \phi(x+y)$. Then we have that $u-\phi_y$ reaches a local minimum at some point $x_y$ in  $B(x_0, \eta)$.
We can show that there exists a sequence $y_k\rightarrow 0$ such that $D\phi_{y_k}(x_{y_k})\neq 0$ for all $k$. Testing the equation at $x_{y_k}$ we get that for $k$ large enough
$$\underset{\substack{z\in B_r(x_0), \\ z\neq x_0}}{\sup}\, -|D\phi(z)|^\gamma\Delta_p^N\phi(z)\geq -|D\phi_{y_k}(x_{y_k})|^\gamma\Delta_p^N\phi_{y_k}(x_{y_k})\geq f(x_{y_k}).$$
Letting $r\to 0$ and $k\to \infty$, we get that 
$$\underset{r\to 0}{\lim}\underset{\substack{x\in B_r(x_0),\\ x\neq x_0}}{\sup}\, -|D\phi(x)|^\gamma\Delta_p^N\phi(x)\geq f(x_0).$$
If such a sequence $(y_k)$ does not exist, then it follows by the property that $x_0$ is the only critical point of $\phi$ in a small neighborhood of $x_0$, that  for  all $y\in B(0, r)$, we must have $x_y + y = x_0$  and 
also that for all $y\in B(0,r)$ and $x\in B(x_0, \eta)$ we have
\begin{equation} 
 u(x_0-y)-\phi(x_0)=u(x_y)-\phi_y(x_y)\leq u(x)-\phi_y(x)=u(x)-\phi(x+y).
 \end{equation}
 That is, 
 \begin{equation}\label{contragin}
 u(x_0-y)-u(x)\leq\phi(x_0)-\phi(x+y).
 \end{equation}
 
For a vector $z\in B(0,r)$ and a unit vector $e_i$, taking $x=x_0+ z$   and $y=-z-t e_i$ (resp. $x=x_0+z+t e_i$ and $y=-z$) in \eqref{contragin}, dividing by $t>0$ and taking $\limsup$ (resp. $\liminf$),
we see that the directional derivatives of $u$ near $x_0$ vanish
everywhere in the neighborhood. This implies that $u$ is constant around
$x_0$, which is a contradiction.
\end{proof}

We can see from the proof that the three definitions are  also equivalent when $\gamma=0$.
When the operator is continuous ($\gamma>0$),  then Definitions  \ref{def1} and  \ref{def2} are  equivalent with
the usual definition of viscosity solution (see \cite[Lemma 2.1]{dfqsing}).

Next we observe that the singular case can be reduced to the study of the normalized $p$-Laplacian problem with a lower order term.
\begin{lemma}\label{reducsing}
Let  $\gamma \in (-1, 0)$. Assume that $u$ is a viscosity solution to \eqref{genpl}. Then $u$ is  a viscosity solution to 
\begin{equation}
-\Delta_p^N u=f|Du|^{-\gamma}.
\end{equation}
\end{lemma}

\begin{proof}
It is sufficient to prove the supersolution property since
the sub-solution property is similar. 
We give the proof for $p>2$ the other case being analogous. 
Let $\phi$ be a test function touching $u$ strictly from below at $x_0$. 
The case $D\phi(x_0)\neq 0$ is immediately clear, so we focus on the case $D\phi(x_0)=0$. 

We distinguish two cases: either $D^2 \phi(x_0)$ is invertible or not. If it is, then $x_0$ is the only critical point of $\phi$ in a small ball around $x_0$. In this case using  the definition of $x_0$, we have that for all $\eps>0$ there exist $r_0$ such for $r\leq r_0$ there exist $x_r\to x_0$ such that $D\phi(x_r)\neq 0$ and 
$$-|D\phi(x_r)|^\gamma\Delta_p^N \phi(x_r)\geq f(x_0)-\eps,$$
hence
$$-\Delta_p^N\phi(x_r)\geq (f(x_0)-\eps)|D\phi(x_r)|^{-\gamma}.$$

It follows that 
$$-\Delta \phi(x_r)-(p-2)\lambda_{\min} (D^2\phi(x_r))\geq (f(x_0)-\eps)|D\phi(x_r)|^{-\gamma}.$$
Letting $\eps,r \to 0$, we get 
 $$-\Delta \phi(x_0)-(p-2)\lambda_{\min} (D^2\phi(x_0))\geq 0=f(x_0)|D\phi(x_0)|^{-\gamma}.$$

Now we suppose that $D^2\phi(x_0)$ is not invertible. Then we take  a symmetric matrix $B$, semi-positive definite
such that $D^2 \phi(x_0)-\eta B$ is invertible, for all $\eta >0$. We consider the test function 
$$\phi_\eta (x)=\phi(x)-\eta (x-x_0)B(x-x_0).$$
Applying the previous argument we have 
$$-\Delta \phi_\eta(x_0)-(p-2)\lambda_{\min} (D^2(\phi_\eta(x_0))\geq 0=f(x_0)|D\phi_\eta(x_0)|^{-\gamma}.$$
Passing to the limit $\eta\to 0$ we get the desired result.
\end{proof}

The following lemma will be useful in the next section concerning the proof of the improvement of flatness.
\begin{lemma}\label{limitequequiv}
Let $\gamma>-1$ and $p>1$. Assume that $w$ is a viscosity solution of 
$$-|Dw+q|^\gamma\left[\Delta w+(p-2) \dfrac{\langle D^2 w (Dw+q), (Dw+q)\rangle}{|Dw+q|^2}\right] =0.$$
Then $w$ is a viscosity solution of 
$$-\Delta w-(p-2)\dfrac{\langle D^2 w (Dw+q), (Dw+q)\rangle}{|Dw+q|^2}=0.$$
\end{lemma}
\begin{proof}
We can reduce the problem to the case $q=0$ since $v=w-q\cdot x$ solves $-|Dv|^\gamma\Delta_p^Nv=0$.
It is sufficient to prove the super-solution property since
the sub-solution property is similar. 
Let $\phi$ be a test function touching $u$ strictly from below at $x_0$. 
We use the fact that for the homogeneous case and for $1<p<\infty$, it is enough to test with functions such that $D\phi(x_0)\neq 0$ (see \cite{juutinenlm01,kamamik2012}) and the statement follows.
\end{proof}

\section{Local H\"older regularity of the gradient}\label{chapter3}
In this section we give a  proof for Theorem \ref{th1}. We assume that $\gamma\in  (-1, \infty)$, $p>1$ and $f\in L^{\infty}(\Omega)\cap C(\Omega)$,
and we want to show that there exists $\alpha=\alpha(p,n, \gamma)>0$ such that any viscosity solution $u$ of
  \eqref{genpl} is in $C^{1,\alpha}_{\text{loc}}(\Omega)$, and for any $\Omega' \subset \subset \Omega$,
$$
[u]_{C^{1,\alpha}(\Omega')} \le C=C \left(p,n,\gamma,d,d',||u||_{L^\infty(\Omega)},\norm{f}_{L^\infty(\Omega)} \right),
$$
where $d=\diam(\Om)$ and $d'=\dist(\Om', \partial \Om)$. \\

A standard method  to investigate the  regularity of solutions is through their approximations by linear functions.
The goal is to get good estimates of the error of these approximations using a compactness method and the scaling properties of the equation.
 We will use the following characterization of $C^{1,\alpha}$ functions: there exists a positive constant $C$ such that for any $x\in \Omega$ and $r>0$, there exists a vector $l$ for which  
 \begin{equation}\label{mainstep}
\underset{y\in B_r(x)}{\osc} (u(y)-u(x)-l\cdot (x-y))\leq Cr^{1+\a}.
\end{equation}

By rescaling we may assume $\osc{u}\leq 1$, and by proceeding by iteration it is enough to find $\rho\in(0,1)$ such that inequality \eqref{mainstep} holds true for $r=r_k=\rho^k$,  $l=l_k$  and $C=1$. The balls $B_r(x)$ for $x\in \Omega$ and $r<\text{dist}\, (x,\partial \Omega)$ covering the domain $\Omega$, we may work on balls. 
Moreover, by using a translation argument and the following scaling, 
\[
u_r(y)=r^{-\frac{\gamma+2}{\gamma+1}}u(x+ry),
\]
we may work on the unit ball $B_1(0)$ and prove the regularity only at the origin. Considering $u-u(0)$ if necessary, we may suppose that $u(0)=0$. 

We also reduce the problem by rescaling. Let $$\kappa=(2||u||_{L^\infty(B_1)}+(\eps_0^{-1}||f||_{L^\infty(B_1)})^{1/(1+\gamma)})^{-1}.$$ Setting $\tilde{u}=\kappa u$, then $\tilde{u}$ satisfies
\begin{equation*}
-|D\tilde u|^\gamma \Delta_p^N\left(\tilde{u}\right)=\tilde{f}
\end{equation*}
with $$||\tilde{u}||_{L^\infty(B_1)}\leq \frac{1}{2}\qquad\textrm{and}\quad ||\tilde{f}||_{L^\infty(B_1)}\le \eps_0.$$ 

Hence, without loss of generality we may assume that $||u||_{L^\infty(B_1)}\leq 1/2$ and $||f||_{L^\infty(B_1)}\leq \eps_0$, where $\eps_0=\eps_0(p,n, \gamma)$ is chosen later.
 To prove the inductive step, the strategy is  to study the deviations of  $u$ from planes, $w(x)=u(x)-q\cdot x$,  which satisfy in $B_1$
 \begin{equation}\label{devia}
-|Dw+q|^\gamma\left[\Delta w+(p-2) \left\langle D^2w\frac{D w+q}{\abs{D w+q}}, \frac{D w+q}{\abs{D w+q}}\right\rangle\right]=f
\end{equation}
in the viscosity sense. The existence of the vector  $q_{k+1}$ can be reduced to prove an ''improvement of flatness'' for  solutions of \eqref{devia}.  The proof of the improvement of flatness consists of two steps. First we show equicontinuity for uniformly bounded solutions of \eqref{devia} in Lemma \ref{lip1} and Lemma \ref{lipe2}.
Next, by the Arzel\`a-Ascoli theorem we get compactness, which, together with the regularity of the limiting solutions (Lemma \ref{lip2}), allow to show improvement of flatness for solutions of \eqref{devia} in Lemma \ref{flatle} via a contradiction argument. Finally, we prove $C^{1,\alpha}$ regularity for solutions of \eqref{genpl} in Lemma \ref{lemiter} by using Lemma \ref{flatle} and iteration.

\subsection{Equicontinuity for deviations from planes}
First, we need to prove some compactness result for the deviations from planes. We will provide H\"older regularity results  independently of $q$ using different arguments for large and small slopes. 
In the next lemmas we use the following notation for \emph{Pucci operators}:
\[
\mathcal{M}^+(X):=\sup_{A\in \mathcal{A}_{\lambda,\Lambda}} -\tr(AX)
\]
and
\[
\mathcal{M}^-(X):=\inf_{A\in \mathcal{A}_{\lambda,\Lambda}} -\tr(AX),
\]
where $\mathcal{A}_{\lambda,\Lambda}\subset S^n$ is a set of symmetric $n\times n$ matrices whose eigenvalues belong to $[\lambda,\Lambda]$.

 \begin{lemma}\label{lip1}
Let $\gamma\in (-1, 0]$ and $p\in (1, \infty)$. 	Assume that $|q|^{-\gamma}\norm{f}_{L^\infty(B_1)}\leq a_0(p,n, \gamma)$. Let $w$ be a viscosity solution of \textcolor{blue}{\eqref{devia}} in $B_1$. For all $r\in (0,\frac 34)$ there exist a constant $\beta=\beta(p,n, \gamma)\in (0,1)$, and a constant $C=C(p,n,\gamma)>0$ such that   for all  $x,y\in B_r$,
 	\begin{equation}
 	\begin{split}
 	\abs{ w(x)- w(y)}\le C\left(1+\norm{ w}_{L^\infty(B_1)}+\norm{f}_{L^\infty(B_1)}^{\frac{1}{1+\gamma}}\right) \abs{x-y}^\beta.
 	\end{split}
 	\end{equation}
\end{lemma}
\begin{proof}
For $-1<\gamma\leq0$, we observe that 
\begin{align*}|Dw+q|^{-\gamma}f&\leq C|f||Dw|^{-\gamma}+C|q|^{-\gamma}|f|\\
&\leq |Dw|+C|q|^{-\gamma}|f|+C|f|^{\frac{1}{1+\gamma}}.
\end{align*}
 It follows that $w$ satisfies
 $$\left\{ \begin{array}{ll}\mathcal{M}^+(D^2w)+|Dw|+|g|\geq 0\\
\mathcal{M}^-(D^2w)-|Dw|-|g|\leq 0\end{array}\right.$$
 where $|g(x)|:=C(|f(x)||q|^{-\gamma}+|f(x)|^{\frac{1}{\gamma+1}})$ is a bounded function  (uniformly with respect to $|q|$)  satisfying $\norm{g}_{L^\infty(B_1)}\leq C( a_0+\norm{f}_{L^\infty(B_1)}^{\frac{1}{\gamma+1}})$. Then the Hölder estimate for $w$ follows from  the classical theory of uniformly elliptic equations.
\end{proof}

\begin{lemma}\label{prop1}
Let $\gamma\in (0, \infty)$ and $p\in (1, \infty)$. 	Assume that $|q|\geq \Gamma_0=\Gamma_0(p,n, \gamma)\geq 1$.  Let  $ w$ be a viscosity solution of \eqref{devia} in $B_1$ with $\osc_{B_1}( w)\leq 1$ and $||f||_{L^\infty(B_1)}\leq 1$. For all $r\in (0,\frac 34)$  there exist a constant  $\beta=\beta(p,n,\gamma)\in (0,1)$ and a constant $C=C(p,n,\gamma)>0$ such that   for all  $x,y\in B_r$,
 	\begin{equation}
 	\begin{split}
 	\abs{ w(x)- w(y)}\le C \abs{x-y}^\beta.
 	\end{split}
 	\end{equation}
 	
 \end{lemma}
 For the proof, see Appendix \ref{appi1}.
 
The following lemma provides a uniform H\"older estimate for small slopes.

\begin{lemma}\label{lipe2}
Let $\gamma\in (0, \infty)$ and $p\in(1, \infty)$.  For all $r\in (0,3/4)$, there exist  a constant $\beta=\beta(p,n)\in(0,1)$ and a positive constant $C=C(p,n,\gamma)$ such that any viscosity solution $w$ of \eqref{devia} with $\osc_{B_1}( w)\leq 1$,  $||f||_{L^\infty(B_1)}\leq 1$ and  $|q|\leq \Gamma_0=\Gamma_0(p,n, \gamma)$, satisfies 
\begin{equation}
[w]_{C^{0,\beta}(B_r)}\leq C.
\end{equation}
\end{lemma}

\begin{proof}
If $|D w|\geq 2\Gamma_0 $ then $|Dw+q|\geq ||Dw|-|q||\geq \Gamma_0\geq 1$ and $|Dw+q|^{-\gamma}\leq 1$.  Hence
we have
$$\left\{\begin{array}{ll}\mathcal{M}^+(D^2w)+|f|\geq 0\\
\mathcal{M}^-(D^2w)-|f|\leq 0.\end{array}\right.$$
Using the result of  \cite{imbersillarge} (see also \cite{delarue}), there exists $\beta=\beta(p,n)\in(0,1)$ such that
\[
[w]_{C^{0,\beta}(B_r)}\leq  C=C\left(p,n,r,\underset{B_1}{\osc(w)},\norm{f}_{L^n(B_1)}\right).
\] 
\end{proof}

As a consequence of Lemma \ref{prop1} and Lemma \ref{lipe2}, we have the following uniform Hölder estimate.
\begin{lemma}\label{hihtoma}
Let $\gamma\in (0, \infty)$, $p\in(1, \infty)$ and $f\in L^\infty(B_1)\cap C(B_1)$.  Let $w$ be a viscosity solution to \eqref{devia} with $\osc_{B_1}( w)\leq 1$,  $||f||_{L^\infty(B_1)}\leq 1$. Then, for all $r\in (0,3/4)$, there exist  a constant $\beta=\beta(p,n,\gamma)\in(0,1)$ and a positive constant $C=C(p,n,\gamma)$ such that  $w$ satisfies 
\begin{equation}
[w]_{C^{0,\beta}(B_r)}\leq C.
\end{equation}
\end{lemma}
\subsection{Proof of the improvement of flatness}

The next lemma gives uniform H\"older estimates for the limit equation independent of $|q|$, and is needed in the contradiction argument in the proof of the key Lemma \ref{flatle}, where we show improvement of flatness. 
We refer the reader to \cite[Lemma 3.2]{atparuo}.
  
\begin{lemma}\label{lip2}
	Let $v$ be a viscosity solution of
	  \begin{equation}\label{homeo}
-\Delta v-(p-2) \left\langle D^2v\frac{D v+q}{\abs{D v+q}}, \frac{D v+q}{\abs{D v+q}}\right\rangle=0\quad\text{in}\quad B_1
\end{equation}
	  with $\underset{B_1}{\osc{v}}\leq 1$. For all $r\in (0,\frac12]$, there exist constants $ C_0=C_0(p,n)>0$  and $\beta_1=\beta_1(p,n)>0$ such that 
	\begin{equation}
[v]_{C^{1,\beta_1}(B_{r})}\leq  C_0.\qedhere
	\end{equation}
\end{lemma}  

\begin{lemma}\label{flatle} 
Let $\gamma\in (-1, \infty)$ and $p\in (1, \infty)$. When $\gamma<0$ suppose that $\norm{f}_{L^\infty(B_1)}|q|^{-\gamma}\leq a_0(p,n, \gamma)$.
Then there exist $\eps_0=eps_0(p,n, \gamma)\in(0,1)$ and $\rho=\rho (p,n,\gamma)\in(0,1)$ such that, for any $q\in\R^n$ and any viscosity solution $ w$ of \eqref{devia} with $\osc_{B_1}( w)\leq 1$,  $||f||_{L^\infty(B_1)}\leq \eps_0$ and $|q|^{-\gamma}\norm{f}_{L^\infty(B_1)}\leq c_0(p,n, \gamma)\eps_0$ when $\gamma<0$, there exists $q'\in\Rn$ with $|q'|\leq \bar C(p,n, \gamma)$ such that
$$\underset{x\in B_{\rho}}{\osc}\,( w(x)-q'\cdot x)\leq \frac{1}{2}\rho.$$
\end{lemma}

\begin{proof}
We use a contradiction argument. Assume that there exist a sequence of functions $(f_j)$ with $||f_j||_{L^\infty(B_1)} \rightarrow 0$, a sequence of vectors $(q_j)$ such that $|q_j|^{-\gamma}||f_j||_{L^\infty(B_1)} \rightarrow 0$ when $\gamma<0$ and a sequence of viscosity solutions $( w_j)$ with $\osc_{B_1}(w_j)\leq 1$ of
\begin{equation}\label{holidayst}
-|Dw_j+q_j|^\gamma\left[\Delta  w_j+(p-2) \left\langle D^2w_j\frac{D  w_j+q_j}{\abs{D  w_j+q_j}}, \frac{D w_j+q_j}{\abs{D  w_j+q_j}}\right\rangle\right]=f_j,
\end{equation}
such that, for all $q'\in \Rn$ and any $\rho\in(0,1)$
\begin{equation}\label{contra}
\underset{x\in B_{\rho}}{\osc}( w_j(x)-q'\cdot x)>\frac{\rho}{2}.
\end{equation}
Relying on the compactness result of Lemma \ref{lip1} and  Lemma \ref{hihtoma}, there exists a continuous function $ w_{\infty}$ such that $ w_j\to  w_{\infty}$ uniformly in $B_\rho$ for any $\rho \in (0,3/4)$.
Passing to the limit in \eqref{contra}, we have that for any vector $q'$,
\begin{equation}\label{contrafin}
\underset{x\in B_{\rho}}{\osc} ( w_{\infty}(x)-q'\cdot x)>\dfrac{\rho}{2}.
\end{equation}

We treat separately the cases where the sequence $(q_j)$ is bounded or unbounded. Suppose first that the sequence $(q_j)$ is bounded. Using a compactness argument and relaxed limits,  we extract a subsequence $( w_j)$ converging to a limit $ w_{\infty}$, which satisfies in $B_1$
\begin{equation}\label{unifeq}
-|Dw_\infty+q_\infty|^\gamma\tr\left(\left(I+(p-2)\dfrac{Dw_\infty+q_\infty}{|Dw_\infty+q_\infty|}\otimes\dfrac{Dw_\infty+q_\infty}{|Dw_\infty+q_\infty|}\right )D^2 w_\infty\right) =0  
\end{equation}
in a viscosity sense. (Here $q_j\rightarrow q_\infty$ up to the same subsequence.) 
Using the result of Lemma \ref{limitequequiv}, we know  that viscosity solutions to \eqref{unifeq} are viscosity solutions to \eqref{homeo} and 
using the regularity result of Lemma \ref{lip2}, there exist $\beta_1=\beta_1(p,n)>0$ and $C_0=C_0(p,n)>0$ such that $$\norm{w_\infty}_{C^{1,\beta_1}(B_{1/2})}\leq C_0.$$

If the sequence $(q_j)$ is unbounded, we can extract a converging subsequence from $e_j=\frac{q_j}{|q_j|}$ such that $e_j\rightarrow e_\infty$. Multiplying \eqref{holidayst} by $|q_j|^{-\gamma}$ and passing to the limit, we obtain
\begin{equation}\label{der2}
-\Delta w_{\infty}-(p-2) \left\langle D^2w_{\infty}\,e_{\infty}, e_{\infty}\right\rangle=0\qquad\text{in}\quad B_1,
\end{equation}
with $|e_{\infty}|=1$. Noticing that equation \eqref{der2} can be written as
\[
-\tr{((I+(p-2) e_{\infty}\otimes e_{\infty}) D^2w_\infty)}=0,
\]
we see that equation \eqref{der2} is uniformly elliptic with constant coefficients and linear. Using the regularity result of \cite[Corollary 5.7]{caffarellicabrebook}, there is $\beta_2=\beta_2(p,n)>0$ so that $w_\infty\in C^{1,\beta_2}_\text{\text{loc}}$ and there exists $C_0=C_0(p,n)>0$ such that $\norm{w_\infty}_{C^{1,\beta_2}(B_{1/2})}\leq C_0$. 

We have thus shown that in both cases  $w_\infty\in C^{1,\beta}_\text{\text{loc}}$ for $\beta=\min(\beta_1,\beta_2)>0$. Choose $\rho\in (0, 1/2)$ such that
\begin{equation}
C_0\rho^{\beta}\leq \frac{1}{4}.
\end{equation}

By $C^{1,\beta}_\text{\text{loc}}$ regularity, there exists a vector $k_{\rho}$ with $|k_{\rho}|\leq C(p,n, \gamma)$ such that 
\begin{equation}
\underset{x\in B_{\rho}}{\osc} (w_{\infty}(x)-k_{\rho}\cdot x)\leq C_0\rho^{1+\beta}\leq \frac{1}{4}\rho.
\end{equation}
This contradicts \eqref{contrafin} so the proof is complete.
\end{proof}
Once we have proved the improvement of flatness for deviations from planes, the proof of Theorem \ref{th1} proceeds by standard iteration.

\begin{lemma}\label{lemiter}
Let $\rho$ and $\eps_0\in(0,1)$ be as in Lemma \ref{flatle} and let $u$ be a viscosity solution of \eqref{genpl} with $\gamma> -1$, $p>1$,  $\osc_{B_1} (u)\leq1$ and $||f||_{L^\infty(B_1)}\leq \eps_0$. Then, there exists $\a\in\left(0,\frac{1}{1+\gamma}\right]$ such that for all $k\in\N$, there exists $q_k\in \R^n$ such that
\begin{equation}\label{itera}
\underset{y\in B_{r_{k}}}{\osc} \, (u(y)-q_{k}\cdot y)\leq  r_k^{1+\a },
\end{equation}
where $r_k:=\rho^k$ and $|q_{k+1}-q_k|\leq C r_k^\alpha$ with $C=C(p,n,\gamma)$.
\end{lemma}

\begin{proof}
For $k=0$, the estimate \eqref{itera} follows from the assumption $\osc_{B_1}(u) \leq 1$. Next we take $\alpha\in (0,\min(1,\frac{1}{\gamma+1}))$ such that $\rho^{\alpha}>1/2$.
We assume for $k\geq 0$ that we already constructed $q_k\in \R^n$ such that \eqref{itera} holds true. 
To prove the inductive step $k\rightarrow k+1$, we rescale the solution considering for $x\in B_1$
$$w_k(x)=r_k^{-1-\alpha}\big(u(r_k x)-q_k\cdot (r_k x)\big).$$
By induction assumption, we have $\underset{B_1}{\osc}\,(w_k)\leq 1$, and $w_k$ satisfies
$$-\abs{D w_k+\xi_k}^\gamma\left[\Delta w_k+(p-2) \left\langle D^2w_k\frac{D w_k+\xi_k}{\abs{D w_k+\xi_k}}, \frac{D w_k+\xi_k}{\abs{D w_k+\xi_k}}\right\rangle\right]=f_k,  $$
where  $\xi_k=(q_k/r_k^{\alpha})$ and  $f_k(x)=r_k^{1-(\gamma+1)\alpha}f(r_k x)$ with $\norm{f_k}_{L^{\infty}(B_1)}\leq \eps_0$ since $\alpha<\frac{1}{\gamma+1}$.
Moreover, for $-1<\gamma<0$, we have that $|\xi_k|^{-\gamma}\norm{f_k}_{L^\infty(B_1)}\leq \norm{f}_{L^\infty(B_1)}r_k^{1-\alpha}|q_k|^{-\gamma}\leq\eps_0c_0(p,n, \gamma)$.
Here we used that, by assumption it holds
$$|q_k|\leq \sum_{i=0}^k |q_i-q_{i-1}|\leq C\sum_{i=1}^{k-1}\rho^{\alpha i}\leq c_0^{\frac{-1}{\gamma}}(p,n, \gamma).$$
Using the result of Lemma \ref{flatle}, there exists $l_{k+1}\in\R^n$ such that
$$\underset{x\in B_{\rho}}{\osc}\,(w_k(x)-l_{k+1}\cdot x)\leq \frac{1}{2}\rho.$$   
Setting $q_{k+1}=q_k+ l_{k+1} r_k^{\a}$, we get
\[
\underset{x\in B_{r_{k+1}}}{\osc} \, (u(x)-q_{k+1}\cdot x)\leq \dfrac{\rho}{2} r_k^{1+\a }\leq r_{k+1}^{1+\a}.\qedhere
\]
\end{proof}

Since the estimate \eqref{itera} holds for every $k$, the proof of Theorem \ref{th1} is complete.\\

\begin{remark}
1) For $\gamma\leq 0$, when the boundary data $g\in C^{1, \beta}$ and the domain $\Omega$ has a $C^{1, \beta}$ boundary, the boundary H\"older regularity of the gradient is a direct consequence of the uniform ellipticity of the operator and the result of \cite{ma}.\\
2) For $\gamma>0$, the regularity of the gradient up to the boundary could be obtained by adapting the arguments of \cite{birdem5}.
\end{remark}

\subsection{An alternative proof for the regularity of the gradient when $\gamma\leq p-2$}
In this range we can rely on known results of weak theory.
When $\gamma\leq p-2$, we have that viscosity solutions of \eqref{genpl} are viscosity solution of 
\begin{equation}\label{pls}
-\Delta_pu=f|Du|^{p-2-\gamma}.
\end{equation}
  Since $0\leq p-2-\gamma< p-1$ and $f$ is bounded, using the result of \cite[Theorem 1.4]{ochoa} (see also \cite[Theorem 3.6]{atparuo}) which generalize the result of Julin and Juutinen \cite{julin2012new}, we get that viscosity solutions to \eqref{pls} are also weak solutions to \eqref{pls}. The result of \cite{Tolk84} implies that $u\in C^{1, \alpha}_{loc}(\Omega)$ and
$$
[u]_{C^{1,\alpha}(\Omega')} \le C=C \left(p,n,d, \gamma, d',||u||_{L^\infty(\Omega)},\norm{f}_{L^\infty(\Omega)} \right).
$$  
   Moreover for a bounded domain $\Omega$ with a $C^{1, \beta}$ boundary, when we complement the equation \eqref{genpl} with a boundary condition $g\in  C^{1, \beta}(\overline\Omega)$, then it follows from the result of \cite[Theorem 1]{lie88}, that the viscosity solution is in $C^{1, \alpha}(\overline{\Omega})$, with a norm depending only on $n,p, \Omega, \norm{f}_{L^\infty(\Omega)}, \norm{g}_{C^{1, \beta}(\overline\Omega)}$.

\begin{remark}
In the case $0<\gamma\leq p-2$, it is possible to relax the dependence on $f$ in the local $C^{1,\alpha}$estimate from $L^\infty$-norm to $L^q$-norm for some $q<\infty$. Here we just outline the idea and refer to \cite[Section 4]{atparuo}, where a similar technique was used by relying on the paper of Duzaar and Mingione \cite{DM2010}. 

Fix a viscosity solution $u$ of \eqref{genpl} and a small $\lambda>0$, then consider the approximation process which consists in studying the equation 
\begin{equation*}
\begin{cases}
-\text{\emph{div}}\, ((|Dv_\eps|^2+\eps^2)^{\frac{p-2}{2}}Dv_\eps)
=(f_\eps+\lambda u-\lambda v_\eps)(|Dv_\eps|^2+\eps^2)^{\frac{p-2-\gamma}{2}}&\text{in}\quad \Omega\\
v_\eps=u &\text{in}\quad\partial\Omega.
\end{cases}
\end{equation*}
One can provide uniform Lipschitz estimates for $v_\eps$ depending on the the $L^q$ norm of $f$ for some $q>\max\left(n, \dfrac{p}{\gamma+2},\right)$. Indeed using the fact that $u$ is a weak solution we can control the $L^p$ norm of the gradient by a Caccioppoli inequality, and the fact that the $L^\infty$ norm of $u$ is controlled by the $L^n$ norm of $f$, see the ABP estimate \cite[Theorem 1.1]{dfqsing} for viscosity solutions of \eqref{genpl}.
Combining these estimates with the computations of \cite{DM2010}, we get the uniform estimates which depends on a lower norm of $f$. Then one can prove that
\[
[u]_{C^{1,\alpha}(\Omega'')}\leq C=C\left(p,q,n,d,, \gamma, d'',||u||_{L^\infty(\Om)},\norm{f}_{L^q(\Om)}\right),
\]
where $d=\text{\emph{diam}}\, (\Om)$ and $d''=\text{\emph{dist}}\, (\Om'',\partial \Om')$. Here $\Omega''\subset \subset \Omega'\subset \subset \Omega$.
\end{remark}

\section{ $W^{2,2}$ regularity}\label{chapter4}

In this section we study the integrability properties of the second derivatives of the viscosity solutions to \eqref{genpl}.
The idea is to divide the study into different cases. When $\gamma\leq 0$, we reduce the problem to the case of the normalized $p$-Laplacian with a continuous right hand term. We rely on the uniform ellipticity of the operator and on the Cordes condition (see Theorem \ref{cordes}).
 When $\gamma>0$, the operator is degenerate and we have to use another trick.  Then we study a regularized problem. When $1>\gamma>0$ but still $|\gamma-p-2|\leq \delta$  for some $\delta$ small enough and $p$ close to $2$,  we prove
uniform estimates on the second derivatives of the approximate solutions. This will provide the desired result by passing to the limit problem.

Let us recall some known  results and open problems about the existence and integrability of the second derivatives of solutions for $p$-Laplacian type problems.
The $W^{2,2}$-estimates for elliptic equations with measurable coefficients in smooth domains were obtained by Bers and Nirenberg \cite{bers} in the two dimensional case in 1954, and by Talenti \cite{tal} in any dimensions under the Cordes condition. In \cite{camp1} Campanato established the $W^{2,q}$-estimate for elliptic
equations with measurable coefficients in two dimensions for $q$ close to 2.

 There are also some available results  for $p$-harmonic functions \cite{boja,lindq, uhlenbeck77} based on difference quotient, which assert that the nonlinear expression of the gradient
$$|Du|^{\frac{p-2}{2}}D u\in W^{1,2}_{loc}(\Omega).$$ 
This property and the local boundedness of the gradient guarantee that for $p\in (1, 2)$ the second derivative exists and $u\in W^{2,2}_{loc}(\Omega)$.

For $p>2$, the derivative $ \dfrac{\partial}{\partial x_i} \left( |Du|^{\frac{p-2}{2}} \dfrac{\partial u}{\partial x_j}\right)$  exists but the passage to $D^2 u$ is difficult at the critical points. The existence of second derivatives of $W^{1,p}$-solutions is not clear due to the degeneracy of the problem.  In dimension 2 and for $p\in (1, \infty)$,  $p$-harmonic function are in $C^{1,\alpha}_{loc}(\Omega)\cap W^{3, q}_{loc}(\Omega)$ where $q$ is any number $1\leq q< \dfrac{2}{2-\alpha}$, see \cite[Theorem 1]{iwaniecm89}.

The nonhomogeneous case 
\begin{equation}\label{plsnonhom}
-\Delta_p u=g
\end{equation}
was also treated with partial results. For $g\in L^r(\Omega)$ with $r>\max(2, \dfrac n p)$ and $p\in (1, \infty)$, Lou proved in \cite{lou} that weak solutions satisfy $|Du|^{p-1}\in W^{1,2}_{loc}(\Omega)$. Tolksdorf \cite{Tolk84} proved $W^{2,2}_{loc}$-regularity for $p\in (1,2]$ and $g\in L^{\infty}(\Omega)$, see also \cite{acerbi, lindgren, lindq, pucci,serrin}. The $W^{2,2}_{loc}$ regularity for $2<p<3$ and $g\in W^{1, n}(\Omega)$ was proved in \cite{damscui,scu1,scuin7} by using weighted estimates. Recently, Cellina \cite{cel17} relaxed the regularity assumption of $g\in W^{1,2}(\Omega)$ by adapting a Nirenberg technique. In the case $p\geq 3$ and $f\in W^{1,n}(\Omega)$ strictly bounded away from zero or satisfying certain growth condition, it has been shown that $u\in W^{2,q}_{loc}(\Omega)$ for any $q<\frac{p-1}{p-2}$, see \cite{damscui,scu1,scuin7}.
Recently, for $3\leq p<4$ and $f\in W^{1,2}(\Omega)$,  it has been shown in \cite{cel18}, that, $Du$ belongs to $W^{s,2}_{loc}(\Omega)$ for $0<s<4-p$. For  the fractional differentiability of the gradient, we refer the reader to \cite{kuusiming, santambo, mic} and the references therein.

Global (up to the boundary), full regularity for the second derivatives of
the solutions of \eqref{plsnonhom} with 0 boundary conditions are investigated in \cite{crispo2, crispo,crispogm16, crispo3,sc2}. In these papers, for $p\in (C(q),2)$ and any bounded and sufficiently smooth domain $\Omega$, the authors proved   $W^{2,q}$
regularity for any arbitrarily large $q$, getting as a by product result the H\"older continuity up to the
boundary of the gradient of the solution for any $\alpha <1$.  In particular, if $\Omega$ 
is convex, solutions belong to $W^{2,2}(\Omega)$ for any $1<p\leq 2$. The proofs are based on approximation arguments, the assumption that $p$ is close to 2, and the classical Calder\'on-Zygmund theory.

The restriction $p$ small is fundamental as the example of the functions $|x_1|^\beta$ with $\beta>1$ shows.  These functions are local solutions to $-\Delta_p u=g$  for some $g\in L^\infty_{loc}(\R^n)$ provided that $p$
is large enough, but they fail to be in $W^{2,2}_{loc}(\R^n)$
if $\beta \leq \frac32$.

\subsection*{The Cordes condition for operators in nondivergence form}

Here we recall some available results on the summability of the second derivative for operators in non-divergence form with measurable coefficients.
The first are due to \cite{bers, cordes1,tal} and require that the second order linear operator is close to the Laplacian. The case  with lower order term was also treated in  \cite{CM}.
In \cite{camp1}, Campanato extended the $W^{2,2}$ estimate to $W^{2,q}$ for $q$ sufficiently close to 2.
 We refer the reader to \cite[Theorem 1.2.1, Theorem 1.2.3]{bookdisc} for a review on the Cordes condition.

\begin{theorem}\label{cordes}
Consider a linear operator $L$ defined on the space $W^{2,2}_{loc}(\Omega)$ as 
$$Lv(x):= \overset{n}{\underset{i,j=1}{\sum}} a_{i,j}(x)D_iD_j v(x)$$
where $a_{ij}(x)$ is a symmetric matrix with measurable coefficients satisfying the  ellipticity condition

\begin{equation}
\Lambda_1 |\xi|^2\leq \overset{n}{\underset{i,j=1}{\sum}}\, a_{i,j}(x)\xi_i\xi_j\leq \Lambda_2 |\xi|^2
\end{equation}
for some $0<\Lambda_1<\Lambda_2$ and satisfying {\bf the Cordes condition}:\\
there exists $\delta\in(0, 1]$ such that for a.e $x\in \Omega$

\begin{equation}
\left(\overset{n}{\underset{i,j=1}{\sum}} (a_{ij}(x))^2\right)\leq \dfrac{1}{n-1+\delta}\left(\overset{n}{\underset{i=1}{\sum}} a_{ii}(x)\right)^2.
\end{equation}
Then any strong solution $v\in W^{2,2}_{loc}(\Omega)\cap L^2(\Omega)$ of the equation 
$Lv=f$ in $\Omega$ with $f\in L^2(\Omega)$ satisfies, for any $\Omega'\subset\subset \Omega$
\begin{equation}
\int_{\Omega'} |D^2 v|^2\, dx\leq C\left(\norm{f}^2_{L^2(\Omega)}+\norm{v}^2_{L^2(\Omega)}\right)
\end{equation}
where $C=C(n, \Lambda_1, \Lambda_2, \delta, \Omega', \Omega)$.

Moreover, under the same hypothesis on the operator  $L$, there exist two real numbers $1<q_0<2<B_1$ depending on the ellipticity constants and the dimension $n$, such that for any $q\in (q_0, B_1)$, any strong solution of $L	v=f$ with $f\in L^q(\Omega)$, satisfies for any $\Omega'\subset\subset \Omega$
\begin{equation}
\int_{\Omega'} |D^2 v|^q\, dx\leq C\left(\norm{f}^q_{L^q(\Omega)}+\norm{v}^q_{L^q(\Omega)}\right)
\end{equation}
where $C=C(n,q, \Lambda_1, \Lambda_2, \delta, \Omega', \Omega)$.
\end{theorem}

The Cordes condition is equivalent to the uniform ellipticity condition when $n=2$
and stronger when $n\geq 3$.   As an application, it
was used to prove the second order differentiability of $p$-harmonic functions in \cite{manfrediw}.
As it is often the case, the two dimensional case is an exception, since there is no restriction on $p$. We also mention the result of \cite{lin} which asserts that for uniformly elliptic linear equation with measurable coefficients, there exists a universal $r=r(\Lambda_1, \Lambda_2)>0$ such that any $C^{1, 1}$ solution of $Lv=f$  with $f\in L^n(\Omega)$ satisfies 
$$\norm{v}_{W^{2,r}(\Omega')}\leq C\left(\norm{v}_{L^\infty(\Omega)}+\norm{f}_{L^n(\Omega)}\right),$$ 
where $C=C(n, \Omega, \Lambda_1, \Lambda_2)>0$.

The second results can be found in \cite{caffarellicabrebook} and relies on the smallness of the oscillation of the operator measured in the $L^n$ norm.
\begin{theorem}\label{cafi}
Let $v$ be a bounded viscosity solution in $B_1$ of
$$F(D^2v,x)=f(x).$$
Assume that $F$ is uniformly elliptic with ellipticity constants $\lambda$ and $\Lambda$, $F,f$ are continuous in $x$, $F(0,\cdot)\equiv 0$ and $F(D^2w,x_0)$ has $C^{1,1}$ interior estimates (with constant $c_e$) for any $x_0$ in $B_1$. Suppose that $f\in L^q(\Omega)$ for some $n<q<\infty$.
Then there exist positive constants $\beta_0$ and $C$ depending on $n$, $\lambda$, $\Lambda$, $c_e$ and $q$, such that  if the oscillation
$$\beta(x, x_0):=\underset{M\in \mathcal{S}\setminus \left\{0\right\}}{\sup}\, \dfrac{|F(M,x)-F(M, x_0)|}{\norm{M}}$$
satisfies
$$\left( \frac{\int_{B_r(x_0)}\beta(x, x_0)^n\, dx}{B_r(x_0)}\right)^{1/n}\leq \beta_0$$
for any ball $B_r(x_0)\subset B_1$, then $v\in W^{2, q}(B_{1/2})$ and
$$\norm{w}_{W^{2,q}(B_{1/2})}\leq C\left(\norm{w}_{L^\infty(B_1)}+\norm{f}_{L^q(B_1)}\right).$$

\end{theorem}

\subsection{The case $\gamma\leq 0$  and $p$ close to 2}
Using the result of Lemma \ref{reducsing}, we can reduce the study to the case of the normalized $p$-Laplacian with a  bounded right hand term.
In this case the operator is singular but uniformly elliptic. 
We will thus use the result of Theorem \ref{cordes} for  regularized problems and provide uniform local $W^{2,2}$ estimates for $p$ in the range where the assumptions of Theorem \ref{cordes} are satisfied, $$1<p<3+\dfrac{2}{n-2}.$$

\noindent \emph{Proof of Theorem \ref{th2}.} For the moment, we can prove uniqueness only for $f=0$ (using the weak theory of the standard $p$-Laplacian) or $f$ with a constant sign (the proof is an adaptation of the arguments used in \cite{kamamik2012} using the viscosity theory). To avoid dealing with the problem of uniquess, we use the classical trick of adding a certain zero order term.

Let $u$ be a viscosity solution of equation \eqref{genpl}. For any  $\lambda >0$, the function $u$  is  a viscosity solution of 
\begin{equation*}
-\Delta_p^N u(x)+\lambda u(x)=h(x):=f(x)|Du(x)|^{-\gamma}+\lambda u(x),  \quad x\in \Omega.
\end{equation*}

 Let $\Om'\subset\subset \Om$ with $\Om'$ smooth enough. In the sequel we fix small enough $\lambda>0$ and  a viscosity solution $u$ of \eqref{genpl}. We take smooth functions $f_\eps\in C^1(\Om)\cap L^\infty(\Om)$ converging uniformly to $f$ in $\Om'$.

We consider the following regularized problem,
\begin{align}\label{perti}
\left\{
\begin{array}{ll}
-\Delta v_{\eps}-(p-2) \frac{D^2v_{\eps} Dv_{\eps}\cdot Dv_{\eps}}{| D v_{\eps}|^2+ \eps^2} +\lambda v_\eps=f_\eps(|Dv_\eps|^{2}+\eps^2)^{-\gamma/2}+\lambda u &\quad\text{in}\,\,\Om'\\
v_\eps=u&\quad\text{on}\,\, \partial\Omega'.
\end{array}
\right.
\end{align}
It follows, using the fact that the problem is uniformly elliptic without singularities, that solutions $v_\eps$ are classical $C^{2}$ solutions and solve the equation in the classical sense, see \cite[Theorem 15.18]{gilbargt01} and \cite[Theorem 3.3]{LU68}. From the comparison principle we have an estimate
\[
\norm{v_\eps}_{L^\infty(\Omega')}\leq 2\left(\norm{u}_{L^\infty(\Omega)}+\frac{\norm{f}_{L^\infty(\Omega')}}{\lambda}\right),
\]
In Appendix \ref{appendixb} we also show that 
\begin{align}\label{jatiksaha}
&||Dv_\eps||_{L^\infty (\Omega')}\\
&\leq C(p,n,\gamma,\Omega')(||f||^{\frac{1}{1+\gamma}}_{L^{\infty}(\Omega')}+||v_\eps||_{L^{\infty}(\Omega')}+||\bar h_\eps||_{L^{\infty}(\Omega')},\norm{u}_{W^{1,\infty}(\Omega')})\nonumber\\
& \leq C(p,n,\gamma,\Omega')(||f||^{\frac{1}{1+\gamma}}_{L^{\infty}(\Omega)}++||u||_{L^{\infty}(\Omega)}+||f||_{L^{\infty}(\Omega)}),
\end{align}
where we denoted $\bar h_\eps:=\lambda (u-v_\eps)$.

In the sequel we take $\lambda=1$. Consider the operator 
\begin{equation*}
L_{v_{\eps}}(v):=-\Delta v-(p-2) \frac{D^2v Dv_{\eps}\cdot Dv_{\eps}}{| D v_{\eps}|^2+ \eps^2}.
\end{equation*}
The operator $L_{v_{\eps}}$ is uniformly elliptic with $\Lambda_1=\min (1, p-1)$ and $\Lambda_2=\max(1,p-1)$ and satisfies the Cordes condition with $\delta=\delta(p,n)$ for $1<p<3+\dfrac{2}{n-2}$. We have 
$$L_{v_{\eps}} v_{\eps}= F_\eps$$
where $F_\eps:=\bar h_\eps+f_\eps(|Dv_\eps|^2+\eps^2)^{-\gamma/2}$. Notice that $F_\eps$ is uniformly bounded in $L^2(\Omega')$. Indeed, by using \eqref{jatiksaha} we obtain
\begin{align*}
||F_\eps||_{L^2(\Omega')}&=||\bar h_\eps+f_\eps(|Dv_\eps|^2+\eps^2)^{-\gamma/2}||_{L^2(\Omega')}\\
& \leq C(\gamma)(|||D v_\eps|^{-\gamma}f||_{L^2(\Omega')}+||f||_{L^2(\Omega')})+|| u||_{L^2(\Omega')}+||v_\eps||_{L^2(\Omega')}\\
& \leq C(p,n,\gamma)|\Omega'|^{1/2}(||f||^{\frac{1}{1+\gamma}}_{L^\infty(\Omega)}+||f||_{L^\infty(\Omega)}+||u||_{L^\infty(\Omega)}).
\end{align*}

Theorem \ref{cordes} implies that $D^2 v_{\eps}$ are locally uniformly bounded in $L^2(\Omega')$ and for $\Omega''\subset\subset\Omega'$,
\begin{align*}
\int_{\Omega''} |D^2 v_\eps|^2\, dx&\leq C\left(\norm{F_\eps}_{L^2(\Omega')}+\norm{v_\eps}_{L^2(\Omega')}\right)\\
& \leq C(p,n,\gamma)|\Omega'|^{1/2}(||f||^{\frac{1}{1+\gamma}}_{L^\infty(\Omega)}+||f||_{L^\infty(\Omega)}+||u||_{L^\infty(\Omega)}).
\end{align*}
Moreover $\left\{v_{\eps}\right\}$ are equi-H\"older continuous in $\overline{\Omega'}$ since the boundary data is H\"older continuous and the operator is uniformly elliptic with ellipticity constants independent of $\eps$, see \cite[Theorem 1.2]{dfqsing}. By the Arzelà-Ascoli theorem, $v_\eps\to v_\infty$ uniformly in $\overline{\Omega'}$.
It follows from the relaxed limit stability that $v_\infty$ is a viscosity solution to 
\begin{equation}\label{lifer}
-\Delta_p^ Nv_\infty+\lambda v_\infty= f|Dv_\infty|^{-\gamma}+\lambda u,
\end{equation}
with $v_\infty=u$ on $\partial\Omega'$. By the uniqueness of viscosity solutions of \eqref{lifer} (see \cite[Theorem 6.1]{dfqsing} and \cite[Proposition 2.2]{birdemen1}), we conclude that $v_\infty=u$ and hence the estimate holds for $u$.\qed\\

Notice that for $f\equiv 0$, using the result of Lemma \ref{limitequequiv} and the previous result, we have that for  $\gamma>-1$ and $1<p<3+\dfrac{2}{n-2}$, any viscosity solution of $-|Du|^\gamma \Delta_p^N u=0$ belongs to $W^{2,2}_{loc}(\Omega)$.\\

 Note also that using a regularizing problem and using the result of Theorem \ref{cordes}, we can show that for $1<p<3+\dfrac{2}{n-2}$, $p-2\leq q\leq p$ and $f\in L^\infty(\Omega)$, there exists a weak solution $u$ of $-\Delta_p u=f|Du|^q$ such that $u\in W^{2,2}_{loc}(\Omega)\cap C^{1, \alpha}_{loc}(\Omega)$. This result extends the one stated in \cite{leonori}, where $q\geq p/2$ and $1<p<3$.

In the case $\gamma<0$, D. Li and Z. Li \cite{lil17,lil} studied regularity of viscosity solutions of
\begin{equation}\label{yleinen} 
-|D u|^{\gamma}F(D^2u,Du,u,x)=f, 
\end{equation} where $F$ is uniformly elliptic. This equation includes equation \eqref{genpl}. However, both their result and method are different from ours. They prove that there exists some $\delta>0$ depending on the ellipticity constants, dimension and $\gamma$, such that viscosity solutions of \eqref{yleinen} are globally in $W^{2,\delta}(\Omega)$, and the estimate depends on the $L^n$-norm of $f$. Their method is based on an ABP estimate, barrier function method, touching by paraboloids, a localization argument, and a certain covering lemma, which plays a similar role than the standard Calder\'on-Zygmund cube decomposition lemma.

The classical results of fully non linear elliptic equations require a small oscillation condition on the coefficients. Applying the result of Theorem \ref{cafi}, we would also get uniform $W^{2, q}$ estimates for $p$  close to 2 (using that $\beta(x,x_0)\leq 2|p-2|$).
Nevertheless, the results using the Cordes condition gives a more precise range on the values of $p$ where we are granted that $W^{2, 2}$ estimates hold. In our case, since we have a uniform Lipschitz bound for $u$, Theorem 7.4 of \cite{caffarellicabrebook} gives the existence of $\delta=\delta(p,n,\gamma)$ such that for every solution $u$ of \eqref{genpl} with $-1<\gamma\leq 0$ and $p>1$, we have $u\in W^{2,\delta}_{loc}(\Omega)$, and 
\[
||u||_{W^{2,\delta}(\Omega')}\leq C(p,n,\gamma,\Omega')(||f||^{\frac{1}{1+\gamma}}_{L^\infty(\Omega)}+||u||_{L^\infty (\Omega)}).
\]
This provides an alternative proof for the result of Li and Li \cite{lil17} in the special case of equation \eqref{genpl}

\subsection{The case $\gamma>0$ but close to 0 and $p$ close to 2}
In case that $\gamma>0$, we still have some results for  $\gamma<1$ and $|p-2-\gamma|$ close to 0. 
Consider smooth solutions to the following problem,
\begin{align}\label{pertisa}
\left\{
\begin{array}{ll}
-(|Dv_\eps|^2+\eps^2)^{\frac\gamma 2}\left[\Delta v_{\eps}+(p-2) \frac{D^2v_{\eps} Dv_{\eps}\cdot Dv_{\eps}}{| D v_{\eps}|^2+ \eps^2}\right] +\lambda v_\eps=h_\eps &\qquad\text{in}\,\,\Om'\\
v_\eps=u&\qquad\text{on}\,\, \partial\Omega',
\end{array}
\right.
\end{align}
where $h_\eps:=f_\eps+\lambda u$, with $f_\eps$ smooth and converging locally uniformly to $f$.\\
The problem being uniformly elliptic without singularities and the right hand side being $C^1$-continuous, the function $v_\eps$ are in $C^{2, \alpha}$.
The existence of smooth solutions $v_\eps$ is ensured by the classical theory (see \cite[Theorem 15.18]{gilbargt01}).\\

\noindent \emph{Proof of Theorem \ref{th3}}. \\
\noindent{\bf Step 1}: Uniform Lipschitz estimates.\\
Applying the comparison principle for elliptic quasilinear equation in general form  (see \cite[Theorem 10.1]{gilbargt01}), we have that 
\begin{align}\label{kaali}
\norm{v_\eps}_{L^\infty(\Omega')}&\leq C\left(\norm{u}_{L^\infty(\Omega)}+
\dfrac{\norm{h_\eps}_{L^\infty(\Omega')}}{\lambda}\right)\nonumber\\
& \leq C\left(\norm{u}_{L^\infty(\Omega)}+\frac{\norm{f}_{L^\infty(\Omega')}}{\lambda}\right). \end{align}
Next, if $\gamma\leq p-2$ then $v_\eps$ are solutions to 
$$-\text{div}\,\left((|Dv_\eps|^2+\eps^ 2)^{\frac{p-2}{2}}Dv_\eps\right)=(h_\eps-\lambda v_\eps)  (|Dv_\eps|^2+\eps^ 2)^{\frac{p-2-\gamma}{2}}.$$
 Using that the boundary data $u$ is in $C^{1, \alpha}(\overline{\Omega'})$, it follows from \cite[Theorem 1]{lie88} that $v_\eps $ are uniformly bounded in $C^{1,\alpha}(\overline{\Omega'})$.  We also have a  uniform Lipschitz bound  on $v_\eps$ using the Ishii-Lions method in the case where we lack  the divergence structure ($\gamma>p-2$)
\begin{equation}\label{lanttu} 
\norm{Dv_\eps}_{L^\infty(\Omega')}\leq C(p,n,\gamma,\Omega')(\norm{f}_{L^\infty(\Omega)}+\norm{u}_{L^\infty(\Omega)}).
\end{equation} 
The proof of this estimate is provided in the appendix, see Proposition \ref{posga}.

\noindent{\bf Step 2}: Uniform estimates for the Hessian.\\
Equation \eqref{pertisa} can be regarded as a perturbation of the regularized $\gamma+2$-Laplacian. Indeed, we can  rewrite 
\begin{align*}
&-(|Dv_\eps|^2+\eps^2)^{\frac\gamma 2}\left[\Delta v_{\eps}+(p-2) \frac{D^2v_{\eps} Dv_{\eps}\cdot Dv_{\eps}}{| D v_{\eps}|^2+ \eps^2}\right]\\
&=-\text{div}\, \left((|Dv_\eps|^2+\eps^2)^{\frac\gamma 2} Dv_\eps\right)-(p-2-\gamma)(|Dv_\eps|^2+\eps^2)^{\frac \gamma 2}\frac{D^2v_{\eps} Dv_{\eps}\cdot Dv_{\eps}}{| D v_{\eps}|^2+ \eps^2}\\
&=-\dfrac{p-2}{\gamma}\text{div}\, \left((|Dv_\eps|^2+\eps^2)^{\frac\gamma 2} Dv_\eps\right)+\dfrac{p-2-\gamma}{\gamma}(|Dv_\eps|^2+\eps^2)^{\frac\gamma 2}\Delta v_{\eps}.
\end{align*}

Hence we have 
\begin{align}\label{kissanviikset}
-\text{div}\, \left((|Dv_\eps|^2+\eps^2)^{\frac\gamma 2} Dv_\eps\right)=(p-2-\gamma)(|Dv_\eps|^2+\eps^2)^{\frac \gamma 2}\frac{D^2v_{\eps} Dv_{\eps}\cdot Dv_{\eps}}{| D v_{\eps}|^2+ \eps^2}+g_\eps,
\end{align}
where $g_\eps:=f_\eps+\lambda u-\lambda v_\eps$.

Multiplying \eqref{kissanviikset} by a smooth test function $\phi\in C_0^2(\Omega')$ and integrating by parts, we have
\begin{align*}
&\int_{\Om '}\left( |D v_\eps|^2+\eps^2\right)^{\frac{\gamma}{2}} D v_\eps\cdot D\phi\,dx\\
&\qquad=\int_{\Om '} \left((p-2-\gamma)(|Dv_\eps|^2+\eps^2)^{\frac \gamma 2}\frac{D^2v_{\eps} Dv_{\eps}\cdot Dv_{\eps}}{| D v_{\eps}|^2+ \eps^2}+g_\eps\right)\phi\, dx.
\end{align*}
Choosing $D_s\phi$ instead of $\phi$  for $s\in\left\{1,..., n\right\}$ as a test function and integrating by parts, we get 
\begin{align*}
&\int_{\Om '}\left(|D v_\eps|^2+\eps^2\right)^{\frac{\gamma}{2}} \left[\sum_i D_iD_s v_{\eps}D_i\phi+\gamma \underset{i,j}{\sum}\, \dfrac{D_jv_\eps D_jD_sv_{\eps} \, D_iv_{\eps} D_i\phi}{|D v_\eps|^2+\eps^2}\right]
\, dx\\
&=-\int_{\Om '} D_s \phi\left((p-2-\gamma)(|Dv_\eps|^2+\eps^2)^{\frac \gamma 2}\frac{D^2v_{\eps} Dv_{\eps}\cdot Dv_{\eps}}{| D v_{\eps}|^2+ \eps^2}+g_\eps\right) \,dx\\
&=-\int_{\Om '} D_s \phi(p-2-\gamma)(|Dv_\eps|^2+\eps^2)^{\frac \gamma 2}\frac{D^2v_{\eps} Dv_{\eps}\cdot Dv_{\eps}}{| D v_{\eps}|^2+ \eps^2}\,dx+\int_{\Om '} \phi D_s g_\eps  \,dx.
\end{align*}
Denoting $$\tilde{A}^{\eps}:=I+\gamma\dfrac{Dv_{\eps}\otimes Dv_{\eps}}{|Dv_{\eps}|^2+\eps^2}$$
 and $$H(Dv_{\eps}):=\left(|D v_\eps|^2+\eps^2\right)^{\frac12},$$ we have 
\begin{align*}
&\int_{\Om '} H(Dv_{\eps})^{\gamma }\,\underset{i,j}{\sum}\, \tilde{A}^{\eps}_{ij} D_jD_s v_{\eps} D_i\phi\, dx\\
& =-(p-2-\gamma)\int_{\Om '} D_s\phi H(Dv_\eps)^{ \gamma} \frac{D^2v_{\eps} Dv_{\eps}\cdot Dv_{\eps}}{| D v_{\eps}|^2+ \eps^2} \,dx+\int_{\Om '} \phi D_s g_\eps  \,dx.
\end{align*}
For $0<\beta<1$, taking
\[\phi:=\eta^2 (|Dv_{\eps}|^{2}+\eps^2)^{\frac{-\beta}{2}} D_s v_{\eps}, \]
where $\eta\in C_c^{\infty}$ is a non negative cut-off function,
we have 
\begin{align*}
D_i\phi&=\eta^2 H(Dv_{\eps})^{-\beta}D_iD_s v_{\eps} +2\eta D_i\eta H(Dv_{\eps})^{-\beta}D_sv_{\eps}\\
&\quad-\beta \eta^2\sum_l D_iD_l v_{\eps} D_l v_{\eps} H(Dv_{\eps})^{-\beta-2}D_sv_{\eps}.
\end{align*}
Summing up over $s\in\left\{1,..., n\right\}$, it follows that
\begin{align*}
&I_1+I_2+I_3:=\int_{\Om '} \eta^2H(Dv_{\eps})^{\gamma-\beta}\underset{i,j,s}{\sum} \left(\tilde{A}^{\eps}_{ij} D_jD_s v_{\eps} D_iD_sv_{\eps}\right) dx\\
&-\beta\int_{\Om '} \eta^2H(Dv_{\eps})^{\gamma-\beta-2} \underset{i,j,s}{\sum} \left(\tilde{A}^{\eps}_{ij} D_jD_s v_\eps D_sv_{\eps}\sum_l D_iD_l v_{\eps} D_l v_{\eps}\right)dx\\
&+\int_{\Om '}2 \eta H(Dv_{\eps})^{\gamma-\beta}\underset{i,j,s}{\sum}\left( \tilde{A}^{\eps}_{ij} D_jD_s v_\eps D_sv_{\eps}\right) D_i\eta\, dx\\
=&-(p-2-\gamma)\int_{\Om '}\eta^2\underset{s}{\sum}\,D_sD_s v_\eps H(Dv_{\eps})^{\gamma-\beta} \frac{D^2v_{\eps} Dv_{\eps}\cdot Dv_{\eps}}{| D v_{\eps}|^2+ \eps^2}\,dx\\
&+\beta(p-2-\gamma)\int_{\Om '}\eta^2\underset{s,l}{\sum}\, D_s v_\eps D_sD_l v_\eps D_l v_\eps H(Dv_{\eps})^{\gamma-\beta-2} \frac{D^2v_{\eps} Dv_{\eps}\cdot Dv_{\eps}}{| D v_{\eps}|^2+ \eps^2}\,dx\\
&-2(p-2-\gamma)\int_{\Om '}\eta\sum_s D_s\eta D_s v_{\eps}H(Dv_{\eps})^{\gamma-\beta} \frac{D^2v_{\eps} Dv_{\eps}\cdot Dv_{\eps}}{| D v_{\eps}|^2+ \eps^2}\,dx\\
&+\int_{\Omega'} \sum_sD_s g_\eps D_s v_\eps \eta^2H(Dv_\eps)^{-\beta}\, dx\\
=&:II_1+II_2+II_3+II_4.
\end{align*}
That is, 
$$I_1+I_2=II_1+II_2+II_3+II_4-I_3.$$
The estimates of $I_1, I_2, I_3$ are given as follows.

We have 
\begin{align*}
I_1+I_2&=\int_{\Omega'}\eta^2H(D v_\eps)^{\gamma-\beta}\left[|D^2v_\eps|^2-\dfrac{\beta}{4} |D (|Dv_\eps|^2)|^2H(Dv_\eps)^{-2}\right]\, dx\\
&+\underbrace{\dfrac\gamma 4\int_{\Omega'}\eta^2H(Dv_\eps)^{\gamma-\beta-2} \left[|D (|Dv_\eps|^2)|^2-\beta \bigg|Dv_\eps\cdot D(|D v_\eps|^2|)\bigg|^2H^{-1}\right]\, dx}_{\geq 0}, 
\end{align*}
so that for $\gamma> 0$
\begin{align*}
I_1+I_2&\geq (1-\beta)\int_{\Om '} \eta^2H(Dv_{\eps})^{\gamma-\beta}|D^2 v_\eps|^2\, dx\\
&\quad+
\dfrac{(1-\beta)\gamma}{4} \int_{\Omega'}\eta^2H(Dv_\eps)^{\gamma-\beta-2} |D (|Dv_\eps|^2)|^2\, dx.
\end{align*}

Next, using Young's inequality, it follows that
\begin{align*}
\left|2 H(Dv_\eps)^{\gamma-\beta}\eta \underset{i,j,s}{\sum}\left( \tilde{A}^{\eps}_{ij} D_jD_s v_\eps D_sv_{\eps}\right)D_i\eta\right|&\leq 
\delta_1 |D (|Dv_\eps|^2)|^2\eta^2H(Dv_\eps)^{\gamma-\beta-2}\\
&+ C(p,\delta_1) |D\eta|^2 H(Dv_\eps)^{\gamma-\beta+2},
\end{align*}
and hence 
\begin{align*}
|I_3|&\leq \delta_1 \int_{\Omega'}|D (|Dv_\eps|^2)|^2\eta^2H(Dv_\eps)^{\gamma-\beta-2}\, d x\\
&+C(p,\delta_1) \int_{\Om '}  |D\eta|^2 H(Dv_{\eps})^{\gamma-\beta+2}\, d x.
\end{align*}
To estimate the terms $II_1$ and $II_2$, we will use that
$|D^2 v_\eps Dv_\eps, Dv_\eps|\leq |D^2v_\eps||Dv_\eps|^2$ and $|\Delta v_\eps|\leq \sqrt n|D^2 v_\eps|$. We obtain
\begin{align*}
|II_1|&\leq \int_{\Om '}|p-2-\gamma|\sqrt n|D^2v_\eps|^2\eta^2 H(Dv_{\eps})^{\gamma-\beta}\, d x
\end{align*}
\begin{align*}
|II_2|&\leq\beta (p-2-\gamma)^+
\int_{\Om '}|D^2v_\eps|^2\eta^2 H(Dv_{\eps})^{\gamma-\beta}\, d x.
\end{align*}
Using Young's inequality, we can estimate $II_3$ by  
\begin{equation*}
|II_3|\leq \delta_2\int_{\Om '} |D |Dv_\eps|^2|^2\eta^2 H(Dv_{\eps})^{\gamma-\beta}\, d x+C(p,\delta_2, \gamma) \int_{\Om '}  |D\eta|^2 H(Dv_{\eps})^{\gamma-\beta+2}\, d x.
\end{equation*}
Finally we estimate $II_4$ by 
\begin{equation*}
|II_4|\leq\int_{\Omega'} \norm{H(D v_\eps)}_{L^\infty(\Omega')}^{1-\beta}\eta^2|Dg_\eps|\,dx.
\end{equation*}
Choosing $\delta_1$  and $\delta_2$ small enough, such that 
\begin{equation}
(1-\beta)\dfrac{\gamma}{4}-\delta_1-\delta_2=0,
\end{equation}
we obtain 

\begin{align*} 
(1-\beta)\int_{\Om '} \eta^2H(Dv_{\eps})^{\gamma-\beta}|D^2 v_\eps|^2&\leq\beta (p-2-\gamma)^+
\int_{\Om '}|D^2v_\eps|^2\eta^2 H(Dv_{\eps})^{\gamma-\beta}\, d x\\
&+\int_{\Om '}|p-2-\gamma|\sqrt n|D^2v_\eps|^2\eta^2 H(Dv_{\eps})^{\gamma-\beta}\, d x\\
&+ \int_{\Omega'} \norm{H(D v_\eps)}_{L^\infty(\Omega')}^{1-\beta}\eta^2|Dg_\eps|\,dx\\
&+\int_{\Om '} C(p,n, \gamma) |D\eta|^2 H(Dv_{\eps})^{\gamma-\beta+2}\, d x.\\
\end{align*}
Now, if 
\begin{equation}\label{condi3}
1-\beta-\sqrt n|p-2-\gamma|-\beta(p-2-\gamma)^+=\kappa>0,
\end{equation}
 we get
\begin{align*}
\kappa \int_{\Om '} \eta^2H(Dv_{\eps})^{\gamma-\beta}|D^2 v_\eps|^2&\leq \int_{\Omega'} \norm{H(D v_\eps)}_{L^\infty(\Omega')}^{1-\beta}\eta^2|Dg_\eps|\,dx\\
&+\int_{\Om '} C(p,n, \gamma) |D\eta|^2 H(Dv_{\eps})^{\gamma-\beta+2}\, d x.\\
\end{align*}
We fix $\lambda=1$. By using \eqref{kaali} and \eqref{lanttu}, we have
\[
H(Dv_\eps)\leq C(p,n,\gamma,\Omega')(\norm{f}_{L^\infty(\Omega)}+\norm{u}_{L^\infty(\Omega)})
\]
and
\begin{align*}
\int_{\Omega'}|D g_\eps|\, dx&=\int_{\Omega'}|Df_\eps+\lambda Du-\lambda Dv_\eps|\, dx\\
& \leq \int_{\Omega'}|D f|\, dx+\lambda|\Omega'|(||D u||_{L^\infty (\Omega')}+||D v_\eps||_{L^\infty (\Omega')})\\
& \leq C(p,n,\gamma,|\Omega|)(||f||_{W^{1,1}(\Omega)}+||u||_{L^\infty (\Omega)}+||f||_{L^\infty (\Omega)}).
\end{align*}
We take $\eta$ such that $\eta\equiv 1$ on $\Omega''\subset\subset \Omega'$, $0\leq\eta\leq 1$  and $\norm{D\eta}_{L^\infty(\Omega')}\leq C$.
For any $\beta \in(0, 1)$ and $0<\gamma\leq \beta$, assuming that \eqref{condi3} holds and $f\in W^{1,1}(\Omega)\cap C(\Omega)$, we get that $D^ 2v_\eps$ is uniformly bounded in $L^2(\Omega'')$,
\begin{align}\label{limiitti}
&\int_{\Omega''}|D^2v_\eps|^2\, dx\leq C\left(n,p, \gamma, \beta,d,d', \norm{f}_{W^{1, 1}(\Omega)},\norm{f}_{L^\infty(\Omega)}, \norm{u}_{L^\infty(\Omega)}\right).
\end{align}
If $\gamma\geq 0$ and $p>1$, then using Proposition \ref{appendix1} below together with \cite[Theorem 1]{gripen}, we see that $v_\eps$ are uniformly bounded in $C^{\overline{\alpha}}(\overline\Omega')$. It follows from the Arzelà-Ascoli theorem, that up to a subsequence $v_\eps$ converges to some function $v$. Using the stability result, we have that $v$ is a viscosity solution of 
\begin{equation}\label{uniikki}
-|Dv|^\gamma \Delta^N_p v+\lambda v=f+\lambda u 
\end{equation}
in $\Omega'$ with $v=u$ on $\partial\Omega'$. By uniqueness of solutions of \eqref{uniikki} for $\lambda>0$ (see \cite[Theorem 1.1]{birdem1}), we conclude that $u=v$ in $\Omega'$. Passing to the limit in \eqref{limiitti}, we get the desired local $W^{2,2}$-estimate for $u$, so the proof of Theorem \ref{th3} is complete. \qed

\begin{remark}
For $\gamma=p-2$, $2<p<3$ and $f\in W^{1,1}_{loc}(\Omega)\cap C(\Omega)$, we recover the $W^{2,2}_{loc}(\Omega)$ regularity for weak solutions of the $p$-Poisson problem $-\Delta_p u=f$.   
\end{remark}

\begin{appendix}
\section{Proof of Lemma \ref{prop1}}\label{appi1}
In the proof we denote by $S^n$ the set of symmetric $n\times n$ matrices. For $\eta,\xi\in \R^n$, we denote by $\eta\otimes \xi$ the $n\times n$-matrix for which $(\eta\otimes \xi)_{ij}=\eta_i \xi_j$.  For $n\times n$ matrices we use the matrix norm
\[
||A||:=\sup_{|x|\leq 1}\{|Ax|\}.
\] 

\noindent \emph{Proof of Lemma \ref{prop1}}
	We use the viscosity method introduced by Ishii and Lions in \cite{ishiilions}.
Let $\gamma>0$	and let $w$ be a solution to \eqref{devia} with $\osc_{B_1} (w)\leq 1$ and $\norm{f}_{L^\infty(B_1)}\leq 1$.
We are going to show that $w$ is H\"older in $B_{3/4}$, and this will imply that $w$ is H\"older in any smaller ball $B_\rho$ for $\rho\in \left(0,\frac34\right)$  with the same H\"older constant.
First observe that $\tilde w(x):=w(x)+q\cdot x$ is a solution to \eqref{genpl}.
Proceeding as in the Appendix \ref{appendixb}, we can first show that $\tilde w$ is Lipschitz continuous  and that for $x, y\in Q_{7/8}$, we have
$$|\tilde w(x)-\tilde w(y)|\leq C(p,n,\gamma)(1+\norm{\tilde w}_{L^\infty(B_1)}+\norm{f}_{L^\infty(B_1)})|x-y|.$$
This implies that $w$ is Lipschitz continuous in $B_{7/8}$ and that
$$|w(x)-w(y)|\leq C(p,n,\gamma)(1+\norm{w}_{L^\infty(B_1)}+\norm{f}_{L^\infty(B_1)}+|q|)|x-y|.$$
Hence if $|q|\geq 2+\norm{w}_{L^\infty(B_1)}+\norm{f}_{L^\infty(B_1)}:=\Gamma_0(p,n, \gamma)$, then for $x,y\in B_{7/8}$ it holds
\begin{equation}\label{kolmatit}
|w(x)-w(y)|\leq \bar c_1(p,n,\gamma)|q||x-y|.\end{equation}
Now we will provide uniform Hölder estimates using  the estimate \eqref{kolmatit}.
 We fix $x_0, y_0\in B_{\frac{3}{4}}$, and consider the auxiliary function 
	$$\Phi(x, y):=w(x)-w(y)-L\phi(\abs{x-y})-\frac M2\abs{x-x_0}^2-\frac M2\abs{y-y_0}^2,$$
	where $
	\phi(t)= t^\beta
	$
	for some  $0<\beta<1$. Our goal is to show that $\Phi(x, y)\leq 0$ for $(x,y)\in B_r\times B_r$, where  $r=\frac 34$. We will argue by contradiction, write two viscosity inequalities and combine them in order to get a contradiction using the second order terms and the strict ellipticity in the gradient direction.

	We argue by contradiction and assume that for all $L>0$, $M>0$ and $\beta\in (0,1)$, 
	$\Phi$ has a positive
	maximum. We assume that the maximum is reached at some point $(x_1, y_1)\in \bar B_r\times \bar B_r$.
	Since $w$ is  continuous and bounded, we get 
	\begin{equation} 
	\begin{split}
	M\abs{x_1-x_0}^2\leq 2\norm{w}_{L^\infty(B_1)},\\
	M\abs{y_1-y_0}^2\leq 2\norm {w}_{L^\infty(B_1)}.
	\end{split}
	\end{equation}
	Notice  that $x_1\neq y_1$, otherwise the maximum of $\Phi$ would be non positive.
	Taking $M\geq\norm{w}_{L^\infty(B_1)}\left(\dfrac{32}{r}\right)^2$, we have $|x_1-x_0|< r/8$ and $|y_1-y_0|< r/8$ so that $x_1$ and $y_1$ are in $B_r$.

	Next we apply Jensen-Ishii's lemma to $w(x)-\frac M2\abs{x-x_0}^2$ and  $w(y)+\frac M2\abs{y-y_0}^2$.
		By  Jensen-Ishii's lemma (also known as theorem of sums, see \cite[Theorem 3.2]{crandall1992user}), there exist
		\[
		\begin{split}
		&(\zeta_x,X)\in \overline{\mathcal{J}}^{2,+}\left(w(x_1)-\frac M2\abs{x_1-x_0}^2\right),\\
		&( \zeta_y,Y)\in \overline{ \mathcal{J}}^{2,-}\left(w(y_1)+\frac M2\abs{y_1-y_0}^2\right),
		\end{split}
		\]
		that is 
		\[
		\begin{split}
		&(a,X+MI)\in \overline{\mathcal{J}}^{2,+}w(x_1),\\ &(b,Y-MI)\in \overline{\mathcal{J}}^{2,-}w(y_1),
		\end{split}
		\]
		where ($\zeta_x=\zeta_y$)
		\[
		\begin{split}
		a&=L\phi'(|x_1-y_1|) \frac{x_1-y_1}{\abs{x_1-y_1}}+M(x_1-x_0)=\zeta_x+M(x_1-x_0),\\
		b&=L\phi'(|x_1-y_1|) \frac{x_1-y_1}{\abs{x_1-y_1}}-M(y_1-y_0)= \zeta_y-M(y_1-y_0).
		\end{split}
		\]
		If $L$ is large enough (depending on $M$, actually since $|x_1-x_0|\leq 2$, $ |y_1-y_0|\leq 2$,$ |x_1-y_1|\leq 2$, it is enough that $L>\frac M\beta 2^{4-\beta}$), we have
		\[
		2L\beta \abs{x_1-y_1}^{\beta-1}\geq\abs{a}\geq L\phi'(|x_1-y_1|) - M\abs{x_1-x_0}\ge \frac L2 \beta \abs{x_1-y_1}^{\beta-1}.
		\]
		\[
		2L\beta\abs{x_1-y_1}^{\beta-1}\geq\abs{b}\geq L\phi'(|x_1-y_1|) - M\abs{y_1-y_0}\ge \frac L2 \beta \abs{x_1-y_1}^{\beta-1}.
		\]
\noindent Moreover, by  Jensen-Ishii's lemma  \cite{crandnote}, for any $\tau>0$, we can take $X, Y\in \mathcal{S}^n$ such that for all $\tau>0$ such that $\tau B<I$, we have
		\begin{equation}\label{maineq1}
		-\frac{2}{\tau} \begin{pmatrix}
		I&0\\
		0&I 
		\end{pmatrix}\leq
		\begin{pmatrix}
		X&0\\
		0&-Y 
		\end{pmatrix}\leq \begin{pmatrix} B^\tau& -B^\tau\\
		-B^\tau& B^\tau\end{pmatrix}
		\end{equation}
	where 
		\begin{align*}	
		B=&L\phi''(|x_1-y_1|) \frac{x_1-y_1}{\abs{x_1-y_1}}\otimes \frac{x_1-y_1}{\abs{x_1-y_1}}\\
		&\quad +\frac{L\phi'(|x_1-y_1|)}{\abs{x_1-y_1}}\Bigg( I- \frac{x_1-y_1}{\abs{x_1-y_1}}\otimes \frac{x_1-y_1}{\abs{x_1-y_1}}\Bigg)\\
		&=L\beta\abs{x_1-y_1}^{\beta-2}\left(I+(\beta-2)\frac{x_1-y_1}{\abs{x_1-y_1}}\otimes \frac{x_1-y_1}{\abs{x_1-y_1}}\right)
		\end{align*}
		and $$B^\tau= (I-\tau B)^{-1}B.$$
		For $\tau=\frac{1}{2L\beta\abs{x_1-y_1}^{\beta-2}}$, we have 
		\begin{align*}	
		B^\tau=(I-\tau B)^{-1} B=2L\beta\abs{x_1-y_1}^{\beta-2}\left(I-2\frac{2-\beta}{3-\beta} \frac{x_1-y_1}{\abs{x_1-y_1}}\otimes \frac{x_1-y_1}{\abs{x_1-y_1}}\right).
\end{align*}	
		Notice that  for $\xi=\frac{x_1-y_1}{\abs{x_1-y_1}}$, we have
	
				\begin{equation}\label{oufi}
		\langle B^\tau \xi,\xi\rangle= 2L\beta\abs{x_1-y_1}^{\beta-2}\left(\frac{\beta-1}{3-\beta}\right)<0.
		\end{equation}
	
		Applying the inequality \eqref{maineq1} to any  vector $(\xi,\xi)$ with $\abs{\xi}=1$, we  have that $X- Y\leq 0$ and 
		\begin{equation}\label{gilout}
		\norm{X},\norm{Y}\leq 2L\beta\abs{x_1-y_1}^{\beta-2}.
		\end{equation}
		  The reader can find more details in \cite{crandnote,crandall1992user,ishiilions}.
		 Using the positivity of the maximum of $\Phi$ and the Lipschitz estimate of $w$ (see \eqref{kolmatit}), we have for $0<\beta\leq \frac{1}{4\bar c_1}$,
		\begin{equation}\label{pogba}2\beta L \abs{x_1-y_1}^{\beta-1}\leq 2\beta \frac{|w(x_1)-w(y_1)|}{|x_1-y_1|}\leq 2\beta \bar c_1|q|\leq\frac{ |q|}{2}.
		\end{equation}
		Denoting $\eta_1=a+q$, $\eta_2=b+q$, we get
		\begin{align}\label{koivu}
		2|q|\geq\abs{\eta_1}&\geq \abs{q}-|a|\geq \frac{\abs{q}}{2}\geq 2L \beta \abs{x_1-y_1}^{\beta-1},\nonumber\\
			2|q|\geq\abs{\eta_2}&\geq \abs{q}-|b|\geq \frac{\abs{q}}{2}\geq 2L \beta \abs{x_1-y_1}^{\beta-1}.
		\end{align}
	Notice also that $|\eta_i|\geq |q|/2\geq 1$.
		The  viscosity inequalities read as
		\begin{equation*}
		\begin{split}
		-f(x_1)&\leq |\eta_1|^\gamma\left[ \tr (X+MI)+(p-2)\dfrac{\left\langle(X+MI) (a+q), (a+q)\right\rangle}{|a+q|^2}\right],\\
		-f(y_1)&\geq |\eta_2|^\gamma\left[ \tr(Y-MI)+(p-2) \dfrac{\left\langle(Y-MI) (b+q), (b+q)\right\rangle}{|b+q|^2}\right].
		\end{split}
		\end{equation*}
	 In other words, 
		\begin{equation*}
		\begin{split}
		-f(x_1)|\eta_1|^{-\gamma}&\leq  \tr (A(\eta_1)(X+MI))\\
		f(y_1)|\eta_2|^{-\gamma}&\leq  -\tr (A(\eta_2)(Y-MI))
		\end{split}
		\end{equation*}
		where  for $\eta \neq 0$ $\bar \eta=\dfrac{\eta}{|\eta |}$ and 
		\[A(\eta):= I+(p-2)\ol\eta\otimes \ol\eta.\]
		Adding the two inequalities, we obtain 
	\begin{equation*}
		\left[f(y_1)|\eta_2|^{-\gamma}-f(x_1)|\eta_1|^{-\gamma} \right]\leq  \tr (A(\eta_1)(X+MI))
		-\tr (A(\eta_2)(Y-MI)).
		\end{equation*}
		It results (using that $|\eta_i|^{-\gamma}\leq 1$) that 
		\begin{align}\label{gregory1}
		-2||f||_{L^\infty(B_1)}  \leq  &\underbrace{\tr (A(\eta_1)(X-Y))}_{(I)}\nonumber\\
		&		+\underbrace{tr ((A(\eta_1)-A(\eta_2))Y)}_{(II)}\nonumber\\
		&+\underbrace{M\big[\tr (A(\eta_1))+\tr (A(\eta_2))}_{(III)} \big].
		\end{align}
		
\noindent\textbf{Estimate of (I)}.		Remark that all the eigenvalues of $X-Y$ are non positive. Applying the previous matrix inequality \eqref{maineq1} to the vector $(\xi,-\xi)$ where $\xi:=\frac{x_1-y_1}{|x_1-y_1|}$  and using \eqref{oufi}, 
		we obtain
		\begin{align}\label{camille}
		\langle (X-Y) \xi, \xi\rangle&\leq 4\langle B^\tau \xi,\xi\rangle\leq 8L\beta\abs{x_1-y_1}^{\beta-2}\left(\frac{\beta-1}{3-\beta}\right)<0.
		\end{align}
		This means that  at least one of the eigenvalue of $X-Y$  that we denote by  $\lambda_{i_0}$ is   negative and smaller than $8L\beta\abs{x_1-y_1}^{\beta-2}\left(\frac{\beta-1}{3-\beta}\right)$. The eigenvalues of $A(\eta_1)$ belong to $[\min(1, p-1), \max(1, p-1)]$.	Using \eqref{camille}, it results that 
		\begin{align*}  
		\tr(A(\eta_1) (X-Y))&\leq \sum_i \lambda_i(A(\eta_1))\lambda_i(X-Y)\\
		&\leq \min(1, p-1)\lambda_{i_0}(X-Y)\\
		&\leq \min(1, p-1)8L\beta\abs{x_1-y_1}^{\beta-2}\left(\frac{\beta-1}{3-\beta}\right).
		\end{align*}
		
\noindent\textbf{Estimate of (II)}.		We have
		\begin{align*}A(\eta_1)-A(\eta_2)&=(\ol\eta_1\otimes \ol\eta_1-\ol\eta_2\otimes \ol\eta_2)(p-2)\\	
	&=[(\ol\eta_1-\ol\eta_2)\otimes\ol\eta_1
	-\ol\eta_2\otimes(\ol\eta_2-\ol\eta_1)](p-2)
		\end{align*}
and it follows that
		\begin{align*}
		\tr( (A(\eta_1)-A(\eta_2)) Y)&\leq n\norm{Y}
		\norm{A(\eta_1)-A(\eta_2)}  \\
		&\leq n\abs{p-2}\norm{Y}|\ol\eta_1-\ol\eta_2|\left( |\ol\eta_1|+|\ol\eta_2|\right)\\
		&\leq 2n\abs{p-2}\norm{Y}|\ol \eta_1-\ol\eta_2|.
		\end{align*}
		By using  \eqref{koivu} and $|\eta_1-\eta_2|\leq 4 M$, we have
		\begin{equation*}
		\begin{split}
		\abs{\ol \eta_1-\ol \eta_2}&=
		\abs{\frac{\eta_1}{\abs {\eta_1}}-\frac{\eta_2}{\abs {\eta_2}}}
\le \max\left( \frac{\abs{\eta_2- \eta_1}}{\abs{\eta_2}},\frac{ \abs{\eta_2- \eta_1}}{\abs{\eta_1}}\right)\\
		&\le \frac {16  M}{ \beta L \abs{x_1-y_1}^{\beta-1}}.
		\end{split}
		\end{equation*}

		Combining the previous estimate with \eqref{gilout}, we obtain
	\begin{align*}
	 \tr( (A(\eta_1)-A(\eta_2)) Y)&\leq 64n\abs{p-2}M \abs{x_1-y_1}^{-1}.
	 \end{align*}
		 \textbf{Estimate of (III)}. We have
		$$ M(\tr(A(\eta_1))+\tr(A(\eta_2)))\leq 2Mn\max(1, p-1).$$

		Gathering the previous estimates with \eqref{gregory1}, we get 
		\begin{align*}
		0&\leq 2\norm{f}_{L^\infty(B_1)} + 64n\abs{p-2}M \abs{x_1-y_1}^{-1}\\
		& +2Mn\max(1, p-1)+\min(1, p-1)8L\beta\abs{x_1-y_1}^{\beta-2}\left(\frac{\beta-1}{3-\beta}\right).\\
		\end{align*}
		 Choosing $L$ satisfying $$L\geq C(p,n,\gamma)(M+ \norm{f}_{L^\infty(B_1)}+1)\geq C(p,n,\gamma)(\norm{w}_{L^\infty(B_1)}+\norm{f}_{L^\infty(B_1)}+1),$$ we get
		$$ 0\leq \dfrac{\min(1, p-1)\beta(\beta-1)}{200(3-\beta)} L\abs{x_1-y_1}^{\beta-2}<0,   $$
		which is  a contradiction and   hence $\Phi(x,y)\leq 0$ for $(x,y)\in B_r\times B_{r}$. The desired result follows since for $x_0,y_0\in B_{\frac{3}{4}}$, we have $\Phi(x_0,y_0)\leq 0$, so we get
		\[
		|w(x_0)-w(y_0)|\leq L|x_0-y_0|^\beta.
		\]
		We conclude that $w$ is H\"older continuous  in $B_{\frac34}$ and 
		$$\abs{w(x)-w(y)}\le C(n,p,\gamma)\left(\norm{w}_{L^\infty(B_1)}+\norm{f}_{L^\infty(B_1)}+1\right) \abs{x-y}^\beta.\qedhere$$

\section{Uniform estimates for the approximating problem}\label{appendixb}

\begin{proposition}\label{appendix1}
Let $v_\eps$
be a smooth  solution  of \eqref{pertisa}
with $\gamma\in [0, \infty)$, $p\in (1, \infty)$ and $\eps \in (0,1)$.
Then
there exists a positive constants $\overline{\alpha}>0$ and $C_{\overline{\alpha}}=C_{\overline{\alpha}}(n, p, \gamma, \norm{h_\eps}_{L^\infty(\Omega')}, \norm{v_\eps}_{L^\infty(\Omega')},\norm{u}_{C^\alpha(\Omega')})$
such that for every $x, y\in \Omega'$, we have
\begin{equation}
|v_\eps(x)-v_\eps (y)|\leq C_{\overline{\alpha}}|x-y|^{\overline{\alpha}}.
\end{equation}
\end{proposition}   
\begin{proof}
The uniform interior estimate follows from \cite[Theorem 1.1]{imbabp}, 
\cite[Theorem 1.1]{delarue} and \cite[Corollary 2]{imbersillarge} since the operator $$F(x,s,\eta,X):= -(|\eta|^2+\eps^2)^{\frac\gamma 2}\left[\tr(X)+(p-2) \frac{X \eta\cdot \eta}{| \eta|^2+ \eps^2}\right] +\lambda s-h_\eps $$ satisfies
\begin{eqnarray}
\left. \begin{array}{r} |\eta| \ge 1 \\
F(x,s,\eta,X) \geq 0 \end{array} \right\}
\Rightarrow  
 \mathcal{M}^+ (X) + \lambda s + |h_\eps(x)| \ge 0 \, ,\\
\left. \begin{array}{r} |\eta| \ge 1 \\ 
F(x,s,\eta,X) \leq 0 \end{array} \right\}
\Rightarrow  \mathcal{M}^- (X)  + \lambda s - |h_\eps(x)| \le 0.
\end{eqnarray}
The up to the boundary H\"older estimate follows from \cite[Theorem 1]{gripen}. 

\end{proof}

We are going to make use of the above H\"older estimate and the
Ishii-Lions’ method \cite{crandall1992user}      again
to prove the following Lipschitz estimate.

\begin{proposition}\label{posga}
Let $v_\eps$
be a smooth  solution  of \eqref{pertisa}
with $\gamma\in [0, \infty)$, $p\in (1, \infty)$ and $\eps \in (0,1)$.
Then
there exists a positive constant $C=C(n, p, \gamma, \norm{h_\eps}_{L^\infty(\Omega')}, \norm{v_\eps}_{L^\infty(\Omega')},\norm{u}_{W^{1,\infty}(\Omega')})$
such that for every $x, y\in \Omega'$, we have
\begin{equation}
|v_\eps(x)-v_\eps (y)|\leq C|x-y|.
\end{equation}
\end{proposition}
\begin{proof}
First we provide an interior estimate. Using a scaling and a translation argument combined  with a covering argument, it is enough to prove the Lemma in the unit ball $B_1$ and for $\text{osc}\, v_\eps\leq 1$ in $B_1$.

Like in the proof of the H\"older continuity (see Lemma \ref{prop1}), we fix $x_0,y_0\in B_{3/4}$ and introduce the auxiliary function 
	$$\Phi(x, y):=v_\eps(x)-v_\eps(y)-L\phi(\abs{x-y})-\frac M2\abs{x-x_0}^2-\frac M2\abs{y-y_0}^2,$$
	where $\phi$ is defined below. Our goal is to show that $\Phi(x, y)\leq 0$ for $(x,y)\in B_r\times B_r$, where $r=\frac45$.
	We take
	\[
	\begin{split}
	\phi(t)=
	\begin{cases}
	t-t^{\delta}\phi_0& 0\le t\le t_1:=(\frac 1 {\delta\phi_0})^{1/(\delta-1)}  \\
	\phi(t_1)& \text{otherwise},
	\end{cases}
	\end{split}
	\]
	where $2>\delta>1$ and $\phi_0>0$ is such that  $t_1> 2 $ and $\gamma \phi_02^{\delta-1}\leq 1/4$. 
	
	Then
	\[
	\begin{split}
	\phi'(t)&=\begin{cases}
	1-\delta t^{\delta-1}\phi_0 & 0\le t\le t_1 \\
	0& \text{otherwise},
	\end{cases}\\
	\phi''(t)&=\begin{cases}
	-\delta(\delta-1)t^{\delta-2} \phi_0& 0< t\le t_1 \\
	0 & \text{otherwise}.
	\end{cases}
	\end{split}
	\]
	In particular, $\phi'(t)\in  [\frac34,1]$ and $\phi''(t)<0$ when $t\in (0,2)$.
	
	We argue by contradiction and assume that
	$\Phi$ has a positive
	maximum at some point $(x_1, y_1)\in \bar B_r\times \bar B_r$.
	Since $v_\eps$ is  continuous and
	its oscillation is bounded by 1, we get
	\begin{equation} 
	\begin{split}
	M\abs{x_1-x_0}^2\leq 2\osc_{B_1}{v_\eps}\leq 2,\\
	M\abs{y_1-y_0}^2\leq 2\osc_{B_1}{v_\eps}\leq 2.
	\end{split}
	\end{equation}
	Notice  that $x_1\neq y_1$.
	Choosing $M\geq\left(\dfrac{32}{r}\right)^2$, we have that $|x_1-x_0|< r/16$ and $|y_1-y_0|< r/16$ so that $x_1$ and $y_1$ are in $B_r$.

	Using the fact that $v_\eps$ is H\"older continuous, it follows, adjusting the constants (by  choosing $2M\leq C_{\overline{\alpha}}$), that  
	\begin{equation}\label{kulio}
	\begin{split}
	M\abs{x_1-x_0}\leq C_{\overline{\alpha}}&\abs{x_1-y_1}^{\beta/2},\\
	M\abs{y_1-y_0}\leq C_{\overline{\alpha}}&\abs{x_1-y_1}^{\beta/2}.
	\end{split}
	\end{equation}

		The  Jensen-Ishii's lemma ensures that
		\[
		\begin{split}
		&( \zeta_x,X)\in \overline{\mathcal{J}}^{2,+}\left(v_\eps(x_1)-\frac M2\abs{x_1-x_0}^2\right),\\
		&(  \zeta_y,Y)\in \overline{ \mathcal{J}}^{2,-}\left(v_\eps(y_1)+\frac M2\abs{y_1-y_0}^2\right),
		\end{split}
		\]
		that is 
		\[
		\begin{split}
		&(a,X+MI)\in \overline{\mathcal{J}}^{2,+}v_\eps(x_1),\\ &(b,Y-MI)\in \overline{\mathcal{J}}^{2,-}v_\eps(y_1),
		\end{split}
		\]
		where ($ \zeta_x= \zeta_y$)
		\[
		\begin{split}
		a&=L\phi'(|x_1-y_1|) \frac{x_1-y_1}{\abs{x_1-y_1}}+M(x_1-x_0)= \zeta_x+M(x_1-x_0),\\
		b&=L\phi'(|x_1-y_1|) \frac{x_1-y_1}{\abs{x_1-y_1}}-M(y_1-y_0)= \zeta_y-M(y_1-y_0).
		\end{split}
		\]
		If $L$ is large enough (depending on the H\"older constant  $C_{\overline{\alpha}}$), we have
		\begin{equation}\label{koivu1}
		2L\geq \sqrt{\abs{a}^2+\eps^2},\sqrt{\abs{b}^2+\eps ^2}\geq L\phi'(|x_1-y_1|) - C_{\overline{\alpha}}\abs{x_1-y_1}^{\overline{\alpha}/2}\ge \frac L2.
		\end{equation}
\noindent The Jensen-Ishii's lemma ensures that for any $\tau>0$, we can take $X, Y\in \mathcal{S}^n$ such that 
		\begin{equation}\label{matriceineq1}
		- \big[\tau+2\norm{B}\big] \begin{pmatrix}
		I&0\\
		0&I 
		\end{pmatrix}\leq
		\begin{pmatrix}
		X&0\\
		0&-Y 
		\end{pmatrix}
		\end{equation}
		and
		\begin{align}\label{matineq2}
			\begin{pmatrix}
		X&0\\
		0&-Y 
		\end{pmatrix}
		\le 
		\begin{pmatrix}
		B&-B\\
		-B&B 
		\end{pmatrix}
		+\frac2\tau \begin{pmatrix}
		B^2&-B^2\\
		-B^2&B^2 
		\end{pmatrix}\\
		=D^2\phi (x_1, y_1)+\dfrac{1}{\tau} \left(D^2\phi(x_1, y_1)\right)^2,\nonumber
		\end{align}
		where 
		\begin{align*}	
		B=&L\phi''(|x_1-y_1|) \frac{x_1-y_1}{\abs{x_1-y_1}}\otimes \frac{x_1-y_1}{\abs{x_1-y_1}}\\
		&\quad +\frac{L\phi'(|x_1-y_1|)}{\abs{x_1-y_1}}\Bigg( I- \frac{x_1-y_1}{\abs{x_1-y_1}}\otimes \frac{x_1-y_1}{\abs{x_1-y_1}}\Bigg)
		\end{align*}	
		and 
		\begin{align*}	
		B^2=
		&\frac{L^2(\phi'(|x_1-y_1|))^2}{\abs{x_1-y_1}^2}\Bigg( I- \frac{x_1-y_1}{\abs{x_1-y_1}}\otimes \frac{x_1-y_1}{\abs{x_1-y_1}}\Bigg)\\
		&\quad +L^2(\phi''(|x_1-y_1|))^2 \frac{x_1-y_1}{\abs{x_1-y_1}}\otimes \frac{x_1-y_1}{\abs{x_1-y_1}}.
\end{align*}	
		Observe that  $\phi''(t)+\dfrac{\phi'(t)}{t}\geq 0$, $\phi''(t)\leq 0$ for $t\in (0,2)$ and hence
		\begin{equation}\label{lilou}
		\norm{B}\leq L \phi'(|x_1-y_1|),
		\end{equation}
		\begin{equation}\label{filou}
	 \norm{B^2}\leq L^2\left(|\phi''(|x_1-y_1|)|+\dfrac{\phi'(|x_1-y_1|)}{|x_1-y_1|}\right)^2.
		\end{equation}
	Besides, for $\xi=\frac{x_1-y_1}{\abs{x_1-y_1}}$, we have
	
				\begin{equation*}
		\langle B\xi,\xi\rangle=L\phi''(|x_1-y_1|)<0, \qquad\langle B^2\xi,\xi\rangle=L^2(\phi''(|x_1-y_1|))^2.
		\end{equation*}
		Taking $\tau=4L\left(|\phi''(|x_1-y_1|)|+\dfrac{\phi'(|x_1-y_1|)}{|x_1-y_1|}\right)$,  we obtain for $\xi=\frac{x_1-y_1}{\abs{x_1-y_1}}$,
			\begin{align}\label{mercit}
		\langle B\xi,\xi\rangle +\frac2\tau \langle B^2\xi,\xi\rangle&=L\left(\phi''(|x_1-y_1|)+\frac2\tau L(\phi''(|x_1-y_1|))^2\right)\nonumber\\
		&\leq \dfrac{L}{2}\phi''(|x_1-y_1|)<0 .
		\end{align}
		Applying  inequalities \eqref{matriceineq1} and \eqref{matineq2} to any  vector $(\xi,\xi)$ with $\abs{\xi}=1$, we  have that $X- Y\leq 0$ and $\norm{X},\norm{Y}\leq 2\norm{B}+\tau$. \\
		
		We denote $\eta_1=\sqrt{\abs{a}^2+\eps ^2}$, $\eta_2=\sqrt{\abs{b}^2+\eps^ 2}$, and $g_\eps=f_\eps+\lambda u-\lambda v_\eps=h_\eps-\lambda v_\eps$, and we write the  viscosity inequalities
		\begin{equation}\label{comp1}
		\begin{split}
		-g_\eps(x_1)|\eta_1|^{-\gamma}&\leq  \tr (X+MI)+(p-2)\dfrac{\left\langle(X+MI) a, a\right\rangle}{|\eta_1|^2}\\
		g_\eps(y_1)|\eta_2|^{-\gamma}&\geq  \tr(Y-MI)+(p-2) \dfrac{\left\langle(Y-MI) b,b\right\rangle}{|\eta_2|^2}.
		\end{split}
		\end{equation}
	 We end up with 
		\begin{equation*}
		\begin{split}
			-g_\eps(x_1)|\eta_1|^{-\gamma}&\leq  \tr (A_\eps(a)(X+MI))\\
		g_\eps(y_1)|\eta_2|^{-\gamma}&\leq  -\tr (A_\eps(b)(Y-MI))
		\end{split}
		\end{equation*}
		where  
		\[A_\eps(\eta):= I+(p-2)\dfrac{\eta\otimes\eta}{|\eta|^2+\eps ^2}.\]
		Adding the two inequalities and using that $|\eta_1|^{-\delta}, |\eta_2|^{-\delta}\leq cL^{-\delta}$, we get 
		\begin{equation}\label{comp2}
		-2c\norm{g_\eps}_{L^\infty(B_1)}L^{-\delta} \leq  \tr (A_\eps(a)(X+MI))
		-\tr (A_\eps(b)(Y-MI)).
		\end{equation}
		It follows that 
		\begin{align}\label{gregory}
		-2c\norm{g_\eps}_{L^\infty(B_1)}L^{-\gamma} \leq  &\underbrace{\tr (A_\eps(a)(X-Y))}_{(I)}
		+\underbrace{tr ((A_\eps(a)-A_\eps(b))Y)}_{(II)}\nonumber\\
		&+\underbrace{M\big[\tr (A_\eps(a))+\tr (A_\eps(b))}_{(III)} \big].
		\end{align}
		
\noindent\textbf{Estimate of (I)}.		Observe that all the eigenvalues of $X-Y$ are non positive. Moreover,  applying the previous matrix inequality \eqref{matineq2} to the vector $(\xi,-\xi)$ where $\xi:=\frac{x_1-y_1}{|x_1-y_1|}$  and using \eqref{mercit}, 
		we obtain
		\begin{align}\label{camille12}
		\langle (X-Y) \xi, \xi\rangle&\leq 4\left(\langle B\xi,\xi\rangle+\frac2\tau\langle B^2\xi,\xi\rangle)\right)\nonumber \\
		&\leq 2 L\phi''(|x_1-y_1|)<0.
		\end{align}
		We deduce that  at least one of the eigenvalue of $X-Y$  denoted by  $\lambda_{i_0}$ is   negative and smaller than $2 L\phi''(|x_1-y_1|)$. The eigenvalues of $A_\eps(a)$ belong to $[\min(1, p-1), \max(1, p-1)]$.	Using \eqref{camille12}, it follows  that 
		\begin{align*}  
		\tr(A_\eps(a) (X-Y))&\leq \sum_i \lambda_i(A_\eps(a))\lambda_i(X-Y)\\
		&\leq \min(1, p-1)\lambda_{i_0}(X-Y)\\
		&\leq 2\min(1, p-1) L \phi''(|x_1-y_1|).
		\end{align*}
		
\noindent\textbf{Estimate of (II)}.		We have
		\begin{align*}A_\eps(a)-A_\eps(b)&=\left(\dfrac{a}{\eta_1}\otimes \dfrac{a}{\eta_1}-\dfrac{b}{\eta_2}\otimes \dfrac{b}{\eta_2}\right)(p-2)\\
	&=\left[\left(\dfrac{a}{\eta_1}-\dfrac{b}{\eta_2}\right)\otimes\dfrac{a}{\eta_1}
	+\dfrac{b}{\eta_2}\otimes\left(\dfrac{a}{\eta_1}-\dfrac{b}{\eta_2}\right)\right](p-2).
		\end{align*}
	It follows that 
	$$|A_\eps(a)-A_\eps(b)|\leq 2|p-2|\left|\dfrac{a}{\eta_1}-\dfrac{b}{\eta_2}\right|.$$
		Hence,
		\begin{align*}
		\tr( (A_\eps(a)-A_\eps(b)) Y)&\leq n\norm{Y}
		\norm{A_\eps(a)-A_\eps(b)}  \\
		&\leq 2n\abs{p-2}\norm{Y}\left|\dfrac{a}{\eta_1}-\dfrac{b}{\eta_2}\right|.
		\end{align*}
		We have
		\begin{equation*}
		\begin{split}
		\abs{\dfrac{a}{\eta_1}-\dfrac{b}{\eta_2}}&
\le 2\max\left( \frac{\abs{a- b}}{\abs{\eta_2}},\frac{ \abs{a- b}}{\abs{\eta_1}}\right)\\
		&\le \frac {8C_{\overline{\alpha}}}{ L}\abs{x_1-y_1}^{\overline{\alpha}/2},
		\end{split}
		\end{equation*}
		where we used \eqref{koivu1} and \eqref{kulio}.\\
		
		Using the estimates \eqref{matriceineq1}--\eqref{filou}, we obtain
		\begin{align*}
		\norm{Y}&=\max_{\ol \xi} |\langle Y\ol \xi, \ol \xi\rangle|
		\le 2 |\langle B\ol \xi,\ol \xi \rangle|+\frac4\tau|\langle B^2\ol \xi,\ol \xi \rangle| \\
		&\leq 4L\left( \frac{\phi'(|x_1-y_1|)}{\abs{x_1-y_1}}+ |\phi''(|x_1-y_1|)|\right).
		\end{align*}
		Hence, remembering that $|x_1-y_1|\leq 2$, we end up with
	\begin{align*}
	 \tr( (A_\eps(a)-A_\eps(b)) Y)&\leq 128n\abs{p-2}C_{\overline{\alpha}} \phi'(|x_1-y_1|) \abs{x_1-y_1}^{-1+\overline{\alpha}/2}\\
	 &\quad+128n\abs{p-2}C_{\overline{\alpha}} |\phi''(|x_1-y_1|)|.
	 \end{align*}
		 \textbf{Estimate of (III)}. Finally, we have
		$$ M(\tr(A_\eps(a))+\tr(A_\eps(b)))\leq 2Mn\max(1, p-1).$$

		Putting together the previous estimates with \eqref{gregory} and recalling the definition of $\phi$, we end up with 
		\begin{align*}
		0&\leq 128n\abs{p-2}C_{\overline{\alpha}}\left(\phi'(|x_1-y_1|)  \abs{x_1-y_1}^{\overline{\alpha}/2-1}+  |\phi''(|x_1-y_1|)|\right)\\
		&\quad+2\min(1, p-1) L \phi''(|x_1-y_1|)  + +2Mn\max(1, p-1)-2cL^{-\gamma}\norm{g_\eps}_{L^\infty(B_1)}\\
		&\leq 128n\abs{p-2}C_{\overline{\alpha}}  \abs{x_1-y_1}^{\overline{\alpha}/2-1}+2nM\max(1, p-1)\\
		&\quad +128n\abs{p-2}C_{\overline{\alpha}}\delta(\delta-1)\phi_0\abs{x_1-y_1}^{\delta-2}-2cL^{-\gamma}\norm{g_\eps}_{L^\infty(B_1)}\\
		&\quad-2\min(1, p-1)  \delta(\delta-1)\phi_0L\abs{x_1-y_1}^{\delta-2}.
		\end{align*}
		Taking $\delta=1+\overline{\alpha}/2>1$ and choosing $L$ large enough depending on $p,n, C_\beta$, $\norm{g_\eps}_{L^\infty}$ we get
		$$ 0\leq \dfrac{-\min(1, p-1)\delta(\delta-1)\phi_0}{200} L\abs{x_1-y_1}^{\delta-2}-2cL^{-\gamma}\norm{g_\eps}_{L^\infty(B_1)}<0,   $$
		which is  a contradiction. 
		Hence $\Phi(x,y)\leq 0$ for $(x,y)\in B_r\times B_{r}$. The desired result follows since for $x_0,y_0\in B_{\frac{3}{4}}$, we have $\Phi(x_0,y_0)\leq 0$, we get
		\[
		|v_\eps(x_0)-v_\eps(y_0)|\leq L\phi(|x_0-y_0|)\leq L|x_0-y_0|.
		\]
		
		The uniform Lipschitz estimate up to the boundary of $\Omega'$ follows from a barrier argument as in the proof of \cite[Lemma 2.2]{birdem5}. We use the following notation. $d(\cdot,\partial \Omega')$ is the distance function of $\partial \Omega'$, $$\Omega'_\delta:=\{x\in \Omega'\ \ :\ \ d(x,\Omega')\leq \delta\},$$ and $\bar d(x)=d(x,\partial \Omega')$ in $\Omega'_\delta$ is smooth and bounded in $\Omega'$.

		We construct a barrier function $$\bar u(x)=\bar v(x)+\bar w(x),$$ where $\bar v$ is a solution of
		\begin{equation}\label{kauser}
		\begin{cases}
		\mathcal{M}(D^2(\bar v))=0\quad &\text{in}\ \ \Omega'\\
		\bar v=u\quad &\text{on}\ \ \partial \Omega',
		\end{cases}
		\end{equation} 
		and $$\bar w_\delta(x)=\frac{C}{\delta}\frac{\bar d(x)}{1+\bar d(x)^\nu}\quad \text{with}\ \ \nu\in (0,1).$$ 
		Recall that the  Pucci operators are convex, uniformly elliptic and depend only on the Hessian. The boundary condition in \eqref{kauser} belonging to $C^{1,\alpha}(\overline {\Omega'})$, we get that $\bar v$ is in $C^2(\Omega')\cap C^{1,\alpha}(\overline{\Omega'})$. Moreover, $\norm{\bar v}_{L^\infty(\Omega')}\leq \norm{u}_{L^\infty(\Omega')}$ and $$\norm{D \bar v}_{L^\infty(\Omega')}\leq C(p,n,\Omega')\norm{D u}_{L^\infty(\Omega')}.$$ 
		
		Taking $$C\geq (1+|\Omega'|^\nu)(\norm{v_\eps}_{L^\infty(\Omega')}+\norm{u}_{L^\infty(\Omega')}),$$ we get $$\bar u\geq v_\eps\quad \text{on}\ \ \partial\Omega'_\delta. $$

		Proceeding as in \cite{birdem5}, we can find $$\delta_0=\delta_0(p,n,\Omega',\norm{f}_{L^\infty(\Omega)},\norm{u}_{W^{1,\infty}(\Omega')},\gamma)>0$$ such that 
		
		$$(|D\bar v+D\bar w_\delta|^2+\eps^2)^{\gamma/2}\mathcal{M}(D^2\bar w)\leq -2\norm{f}_{L^\infty(\Omega)}-\lambda(\norm{u}_{L^\infty(\Omega')} $$ on $\Omega'_\delta$ for any $\delta\leq \delta_0$.
		
		It follows that in $\Omega'_\delta$,
		\begin{align*}
		(|D \bar u|^2+\eps^2)^{\gamma/2}& \left(\Delta \bar u+(p-2)\frac{\langle D^2 \bar u D \bar u,D \bar u \rangle}{|D \bar u|^2+\eps^2}\right)\\
		& \leq (|D \bar u|^2+\eps^2)^{\gamma/2}\mathcal{M}(D^2\bar u)\\
		& \leq (|D \bar v+D\bar w|^2+\eps^2)^{\gamma/2}\mathcal{M}(D^2 \bar w)\\
		& \leq 2\norm{f}_{L^\infty(\Omega)}-\lambda(\norm{u_\eps}_{L^\infty(\Omega')}+\norm{\bar v}_{L^\infty(\Omega')})+\lambda \bar w \\
		&\leq -f_\eps-\lambda(u-\bar u).
		\end{align*}
		By comparison principle, we have $u_\eps\leq \bar u$ in $\Omega'_{\delta}$, and the Lipschitz estimate follows.
\end{proof}

\begin{proposition}\label{appendixkolmonen}
Let $v_\eps$
be a smooth  solution  of \eqref{perti}
with $\gamma\in (-1,0]$, $p\in (1, \infty)$ and $\eps \in (0,1)$.
Then
there exists a positive constant $C=C(n, p, \gamma, \Omega', \norm{h_\eps}_{L^\infty(\Omega')}, \norm{v_\eps}_{L^\infty(\Omega')},\norm{u}_{W^{1,\infty}(\Omega')})$
such that for every $x, y\in \Omega'$, we have
\begin{equation*}
|v_\eps(x)-v_\eps (y)|\leq C|x-y|.
\end{equation*}
\end{proposition}

\begin{proof}
It follows from Lemma \ref{lipe2} that solutions of \eqref{perti} are H\"older continuous with uniform H\"older bound $$C(p,n,\gamma)(||v_\eps||_{L^\infty(\Omega')}+||\overline{h}_\eps||_{L^\infty(\Omega')}+||f||^{\frac{-1}{1+\gamma}}_{L^\infty(\Omega)}+\norm{u}_{L^\infty(\Omega)})$$ The proof proceeds similarly to the previous proof of Proposition \ref{posga}. The main changes occur when writing the viscosity inequalities: Now we write
\begin{equation*}
		\begin{split}
			-\overline{h}_\eps(x_1)-f_\eps(x_1)|\eta_1|^{-\gamma}&\leq  \tr (D(a)(X+MI))\\
		\overline{h}_\eps(y_1)+f_\eps(y_1)|\eta_2|^{-\gamma}&\leq  -\tr (D(b)(Y-MI)).
		\end{split}
		\end{equation*}
		and adding the two inequalities we have 	
		\begin{equation*}
		-2||\overline{h}_\eps||_{L^\infty (\Omega')}-2c\norm{f_\eps}_{L^\infty(\Omega')}L^{-\delta} \leq  \tr (D(a)(X+MI))
		-\tr (D(b)(Y-MI)).
		\end{equation*}	
		By choosing $$L\geq c(||v_\eps||_{L^\infty(\Omega')}+||\overline{h}_\eps||_{L^\infty(\Omega')}+||f||^{\frac{-1}{1+\gamma}}_{L^\infty(\Omega')}),$$
		the proof is completed as the proof of Proposition \ref{posga}.
		
		The uniform up to the boundary Lipschitz estimate follows from \cite{silvestres14} due to the regularity of the boundary data and the uniform ellipticity of the operator and the subquadratic growth with respect to the gradient  of the right hand term of the equation.

\end{proof}

\end{appendix}
\bibliographystyle{siam}

\def\cprime{$'$} \def\polhk#1{\setbox0=\hbox{#1}{\ooalign{\hidewidth
  \lower1.5ex\hbox{`}\hidewidth\crcr\unhbox0}}}

\end{document}